\documentclass[12pt]{article}
\usepackage[DIV12]{typearea} 
\usepackage[utf8]{inputenc}
\usepackage{verbatim}
\usepackage{amsmath}
\usepackage{amsthm}
\usepackage{enumerate}
\usepackage{amssymb}
\usepackage{amsfonts}
\usepackage{textcomp}
\usepackage{stmaryrd}
\usepackage[all]{xy}
\DeclareMathOperator*{\coker}{coker}
\DeclareMathOperator{\rk}{rk}
\DeclareMathOperator{\Pic}{Pic}
\DeclareMathOperator{\Ext}{Ext}
\DeclareMathOperator{\Hom}{Hom}

\DeclareMathOperator{\tr}{tr}
\DeclareMathOperator{\At}{At}
\DeclareMathOperator{\ch}{ch}
\DeclareMathOperator{\cH}{H}
\newcommand{\id}{\mathrm{id}}
\newcommand{\ta}[1]{#1^{[2]}}
\newcommand{\tb}[1]{#1^{[3]}}
\newcommand{\tn}[1]{#1^{[n]}}
\newcommand{\tm}[1]{#1^{[n-1]}}

\newcommand{\nt}[1]{#1^{[n-1,n]}}
\newcommand{\mc}[1]{\mathcal{#1}}
\newcommand{\tamc}[1]{\ta{\mc{#1}}}
\newcommand{\tbmc}[1]{\tb{\mc{#1}}}
\newcommand{\tnmc}[1]{\tn{\mc{#1}}}
\newcommand{\tmmc}[1]{\tm{\mc{#1}}}

\DeclareMathOperator*{\NS}{NS}
\newcommand{\ZZ}{\mathbb{Z}}
\newcommand{\Km}{\mathrm{Km}}
\newcommand{\PP}{\mathbb{P}}
\newcommand{\CC}{\mathbb{C}}
\newcommand{\QQ}{\mathbb{Q}}
\DeclareMathOperator*{\im}{im}
\newcommand{\km}[1]{{#1^{[\Km]}}}
\newcommand{\kmmc}[1]{\mc{#1}^{\Km}}
\newcommand{\kmc}[1]{\mc{#1}^{K_3}}
\DeclareMathOperator{\td}{td}

\newcommand{\Sy}[1]{\mathfrak{S}_{#1}}
\newcommand{\corres}{\mathrel{\widehat{=}}}

\makeatletter \renewenvironment{proof}[1][\proofname]
{\par\pushQED{\qed}\normalfont
\topsep6\p@\@plus6\p@\relax\trivlist \item[\hskip 1.5em \itshape#1\@addpunct{.}]\ignorespaces}{\popQED\endtrivlist\@endpefalse} \makeatother 
\newtheoremstyle{thm}
  {1.5\topsep}   
  {1.5\topsep}   
  {\itshape}  
  {0pt}       
  {\bfseries} 
  {.}         
  {5pt plus 1pt minus 1pt} 
  {} 
\newtheoremstyle{defi}
  {1.5\topsep}   
  {1.5\topsep}   
  {\normalfont}  
  {0pt}       
  {\bfseries} 
  {.}         
  {5pt plus 1pt minus 1pt} 
  {} 
\newtheoremstyle{rem}
  {1.5\topsep}   
  {1.5\topsep}   
  {\normalfont}  
  {0pt}       
  {\bfseries} 
  {.}         
  {5pt plus 1pt minus 1pt} 
  {} 

\theoremstyle{thm}
\newtheorem{thm}{Theorem}[section]
\newtheorem*{thm*}{Theorem}
\newtheorem{prop}[thm]{Proposition}
\newtheorem*{prop*}{Proposition}
\newtheorem{lem}[thm]{Lemma}

\newtheorem{cor}[thm]{Corollary}
\theoremstyle{defi}
\newtheorem{defi}[thm]{Definition}
\newtheorem{exa}[thm]{Example}
\theoremstyle{rem}
\newtheorem{rem}[thm]{Remark}

\begin{document}

\title{Tautological Sheaves: Stability, Moduli Spaces and Restrictions to Generalised Kummer Varieties}
\author{Malte Wandel\\\\Leibniz Universität Hannover\\e-mail: wandel@math.uni-hannover.de}

\maketitle
\begin{abstract}
Results on stability of tautological sheaves on Hilbert schemes of points are extended to higher dimensions and transferred to abelian surfaces and to the restriction of tautological sheaves to generalised Kummer varieties. This provides a big class of new examples of stable sheaves on higher dimensional irreducible symplectic manifolds. Some aspects of deformations of tautological sheaves are studied.
\end{abstract}
\tableofcontents
\setcounter{section}{-1}
\section{Introduction}
Moduli spaces of sheaves on symplectic surfaces play an important role in the construction of irreducible holomorphic symplectic (short: irreducible symplectic) manifolds. A natural question that arises is the following: Can we iterate this process? That is, can we construct new examples of irreducible symplectic manifolds using moduli spaces of sheaves on known higher dimensional irreducible symplectic manifolds such as Hilbert schemes of points on $K3$ surfaces or generalised Kummer varieties? Certainly it is difficult to answer this question in this generality. On the other hand almost no examples of stable sheaves on higher dimensional irreducible symplectic manifolds had been encountered. In \cite{Sch10} the first example of a rank two stable vector bundle on the Hilbert scheme of two points on a $K3$ surface was found. In \cite{Wan12} this result was drastically generalised continuing along the following concept: Start with a stable sheaf on a $K3$ surface, transfer this sheaf to the Hilbert schemes of points using the universal property of the latter and obtain what is called a tautological sheaf and, finally, prove its stability.

This article is to be understood as a sequel to \cite{Wan12}. We further extend its results and transfer them to the case of abelian surfaces. 
The Hilbert scheme of points on an abelian surface contains as a closed subvariety the generalised Kummer variety associated with the surface. We study the restriction of tautological sheaves to Kummer varieties and obtain non-trivial examples of stable sheaves on these manifolds.


The paper is organised as follows: We begin in Section 1.1 by collecting known results on the geometry of hilbert schemes of points on a surface and prove a few technical lemmata used in the sequel. Next, in Section 1.2 we introduce the main objects of this article, the tautological sheaves. In Section 1.3 we introduce polarisations on the Hilbert schemes and compute the slopes of tautological sheaves with respect to these polarisations.
In Chapter 2 we analyse destabilising subsheaves of tautological sheaves on Hilbert schemes of three or more points.
The case of abelian surfaces is treated in Chapter 3. We start by computing the Picard group and intersection numbers for the Hilbert schemes of two (Section 3.1) resp.\ three (Section 3.2) points on an abelian surface. Next we transfer the results on stability of tautological sheaves from regular surfaces to abelian surfaces (Sections 3.3 and 3.4). Finally, we prove stability of the restriction of certain tautological sheaves to the Kummer surface (Section 3.5) and the generalised Kummer variety of dimension four (Section 3.6).
We conclude the paper by studying deformations of tautological sheaves in Chapter 4. We show that the moduli spaces of tautological sheaves can be singular in Section 4.2 and investigate in which way we may deform tautological sheaves together with the underlying manifold in Section 4.3.

\subsection*{Notations and conventions}
\begin{itemize}
\item The base field of all varieties and schemes is the field of complex numbers.
\item For the intersection product inside the Chow or cohomology ring of a smooth variety we write either $l.m$ or $l\cdot m,$ for classes $l$ and $m$, or we will just use juxtaposition $lm$.
\item A \em polarisation \em on a variety $X$ is the choice of a class $H$ inside the ample cone $\mathrm{Amp}(X)$.
\item A \em sheaf \em on a variety $X$ is a coherent $\mc{O}_X$-module.
\item We write $\cong$ to indicate an isomorphism of abelian groups, vector spaces and varieties. We use $\simeq$ for isomorphisms of sheaves.
\item All functors such as pushforward, pullback, local and global homoorphisms and tensor product are not derived unless mentioned otherwise.
\item Let $\mc{F}$ be a locally free sheaf. A locally free subsheaf $\mc{L}$ is called a \em subbundle \em of $\mc{F}.$ (In literature this is sometimes only used for locally free subsheaves $\mc{L}$ such that the quotient $\mc{F}/\mc{L}$ is locally free, too.)
\item Let $\sigma\colon Y\rightarrow Z$ be a morphism of smooth projective varieties. For all sheaves $\mc{H}$ on $Y$ and $\mc{G}$ on $Z$ there is an adjunction isomorphism of $\CC$-vector spaces
\[\Hom_Y(\sigma^\star\mc{G},\mc{H})\cong\Hom_Z(\mc{G},\sigma_\star\mc{H})\]
which will be denoted by $\sigma^\star\dashv\sigma_\star$ ($\sigma^\star$ is left adjoint to $\sigma_\star$).
\item Let $Y\times Z$ be the product of two varieties $Y$ and $Z$. Denote the projections to the corresponding factors by $\pi_1$ and $\pi_2$. For sheaves $\mc{H}$ on $Y$ and $\mc{G}$ on $Z$ we define
\begin{eqnarray*}
\mc{H}\boxtimes \mc{G}&:=&\pi_1^\star\mc{H}\otimes\pi_2^\star\mc{G}\text{ and}\\[4pt]
\mc{H}\boxplus \mc{G}&:=&\pi_1^\star\mc{H}\oplus\pi_2^\star\mc{G}.
\end{eqnarray*}
They are called \em exterior tensor product \em and \em exterior sum \em of the sheaves $\mc{H}$ and $\mc{G}$. In the case $Y=Z$ and $\mc{H}=\mc{G}$ we also write $\mc{H}^{\boxtimes2}$ and $\mc{H}^{\boxplus2}$. This, of course, generalises to products of more than two varieties.
\item Let $G$ be a finite group acting on a smooth projective variety $X$. Let $f\colon X\rightarrow X/G$ be the quotient. If $\mc{L}$ is a $G$-equivariant line bundle on $X$, the pushforward $f_\star\mc{L}$ inherits the $G$-linearisation. Since $G$ acts trivially on $X/G,$ we can define $\mc{L}^G$ to be the sheaf of $G$-invariant sections of $f_\star\mc{L}$, which is a line bundle on $X/G$.
Conversely, the pullback gives a homomorphisms $f^\star\colon \Pic(X/G)\rightarrow \Pic^G(X)$ to the group of  $G$-linearised line bundles on $X.$ This map is injective and its image coincides with line bundles $\mc{L}$ such that for every $x\in X$ the stabiliser group $G_x$ acts trivially on the fibre $\mc{L}_x$ (cf.\ \cite{KKV89}).
Finally, by taking first Chern classes, for every class $l$ in $\NS X$ we can define a class $l^G\in\NS(X/G)$.
\item Let $Y$ be a smooth projective variety $Y$ and $\mc{E}$ a vector bundle. There is an extension
\[\At(\mc{E})\in\Ext^1(\mc{E},\mc{E}\otimes\Omega_Y)\]
called the \em Atiyah class \em of $\mc{E}.$ It was introduced by Atiyah in \cite{Ati57} as an obstruction class for the existence of connections on principal bundles. This class satisfies the following properties:
\begin{itemize}
\item Denote by $\tr\colon\Ext^1(\mc{E},\mc{E}\otimes\Omega_Y)\rightarrow \cH^1(Y,\Omega_Y)$ the natural trace map. We have
\[\tr \At(\mc{E})=c_1(\mc{E}).\]
\item For any morphism $f\colon Z\rightarrow Y$ there is an induced map $\tilde{f}\colon\Ext^1(\mc{E},\mc{E}\otimes\Omega_Y)\rightarrow \Ext^1(f^\star\mc{E},f^\star\mc{E}\otimes\Omega_Z).$ The Atiyah class satisfies
\[\tilde{f}(\At(\mc{E}))=\At(f^\star\mc{E}).\]
\item If $\mc{E}'$ is another vector bundle on $Y,$ we have
\[\At(\mc{E}\otimes\mc{E}')=\At(\mc{E})\otimes\id+\id\otimes\At(\mc{E}').\]
\end{itemize}
More generally, the Atiyah class can be defined for any object in the derived category by using locally free resolutions. For a discussion of some of the properties of the Atiyah class in a more modern language and in the realm of algebraic geometry the interested reader is referred to \cite[Sect. 10.1]{HL97}.

\pagebreak
\item Furthermore, we use the following notation:

\begin{tabular}{lcl}
$\mc{T}_X$&$-$&tangent sheaf of a variety $X$\\
$\Omega_X^k$&$-$&sheaf of $k$-th order differentials\\
$\omega_X$&$-$&canonical line bundle\\
$\omega_f$&$-$&relative canonical sheaf of a morphism $f$\\
$\td_X$&$-$&Todd class of a variety $X$\\
$\td_f$&$-$&Todd class of the relative tangent bundle of a morphism $f$\\
$c_i$&$-$&$i$-th Chern class\\
$\cH^\ast$&$-$&cohomology ring of a variety\\
&&or collection of all cohomology groups of a sheaf\\
$\NS X$&$-$&N\'{e}ron$-$Severi group of a variety $X$\\
$\Pic X$&$-$&Picard group of $X$\\
$(-)^{\Sy{n}}$&$-$&invariant part with respect to the $\Sy{n}$-action\\
$S^nX$&$-$&$n$-th symmetric product of the variety $X$\\
$S^nV$&$-$&$n$-th symmetric product of the vector space $V$\\
$V^\vee$&$-$&dual of the vector space $V$\\
$\mc{F}^\vee$&$-$&dual of the sheaf $\mc{F}$ (defined as $\mc{H}om(\mc{F},\mc{O})$)\\
$h^i(\mc{F})$&$-$&$\dim_\CC \cH^i(\mc{F})$
\end{tabular}
\end{itemize}

\subsection*{Acknowledgements}
I want to thank Klaus Hulek and David Ploog for their support, their help and for many interesting discussions. Thanks to Manfred Lehn, Hisanori Ohashi and Jesse Kass for their good advice.

\section{Preliminairies}\label{sectionprelim}
Throughout this chapter we consider a smooth projective surface $X$  together with a polarisation $H\in\NS (X).$
\subsection{Geometric Considerations}\label{subsectiongeom}
For $n=2$ the geometry of the Hilbert scheme points on a surface is very well accessible: In fact, $\ta{X}$ is the blowup of the symmetric square $S^2X$ along the diagonal. If $n>2$, the situation is much more delicate. An important fact is that the Hilbert$-$Chow morphism $\rho_n\colon\tn{X}\rightarrow S^nX$ is no longer a global blowup morphism. But following \cite{Beau83}, we see that outside codimension two subschemes $\rho_n$ is the blowup of the big diagonal. This is important, especially if we want to determine the Picard group of the Hilbert scheme. Nevertheless, the geometry of Hilbert schemes has been intensively studied. Let us summarise the most important results.\\

Following \cite[Section 1]{EGL01}, we consider the following diagram:
\[\xymatrix{
& \nt{X}\ar[d]^{\sigma_n} \ar[rr]_{w_n} \ar[drr]_{\psi_n} & & \Xi_n\ar[d]_{p_n}\ar[r]_{q_n}&X\\
& \Xi_{n-1}\subset X\times\tm{X}\ar[ld]^{p_{n-1}}\ar[dr]_{q_{n-1}}& & \tn{X}\ar[rd]_{\rho_n}&&X^n\ar[dl]\\
\tm{X}&&X&&S^nX.
}\]
Here we denote by
\[\Xi_n:=\{(x,\xi)\mid x\in\xi\}\subset X\times\tn{X}\]
the universal subscheme and by
\[\nt{X}:=\{(\xi',\xi)\mid\xi'\subset\xi\}\subset\tm{X}\times\tn{X}\]
the so-called \em nested Hilbert scheme. \em 

We have the flat degree $n$ covering $p_n\colon\Xi_n\rightarrow\tn{X}$ which is, in fact, the restriction of the second projection $X\times \tn{X}\rightarrow \tn{X}$. Furthermore, $\nt{X}$ is ismorphic to the blowup of $X\times\tm{X}$ along the universal subscheme $\Xi_{n-1}$. Denote this blowup morphism by $\sigma_n$ and the projections from $X\times\tm{X}$ to $\tm{X}$ and $X$ by $p_{n-1}$ and $q_{n-1}$, respectively. By \cite[Prop.\ 2.1]{ES98} the second projection $\psi_n\colon \nt{X}\rightarrow \tn{X}$ factors through $\Xi_n$ and from \cite[Prop.\ 3.5.3]{Hai01} it follows that $w_n$ is an isomorphism outside codimension four subschemes. Thus the morphism $\psi_n$ is flat outside codimension four. Finally, we have $q_{n-1}\circ\sigma_n = q_n\circ w_n$\\

We have
\[\Pic^0\tn{X}\cong \Pic^0X\]
and embeddings
\[(-)_{\tn{X}}\colon\NS X\hookrightarrow \NS\tn{X},\hspace{5pt} l\mapsto l_{\tn{X}}:=\rho_n^\star(l^{\boxplus n})^{\Sy{n}}\quad\text{and}\quad(-)_{\tn{X}}\colon\Pic X\hookrightarrow\Pic\tn{X}.\]
Furthermore, there is a class $\delta_n\in\NS\tn{X},$ such that $2\delta_n$ is the class of the divisor consisting of all non-reduced subschemes $\xi\subset X$. There is a line bundle $\mc{O}(\delta_n)$ with first Chern class $\delta_n$ such that its pullback $p_n^\star\mc{O}(\delta_n)$ is the relative canonical sheaf of $p_n$.

\begin{lem}\label{eqhigherninductd}
Let $D_n$ be the exceptional divisor of $\sigma_n.$ We have
\begin{equation*}
\psi_n^\star\mc{O}(\delta_n)\simeq\mc{O}(D_n)\otimes\sigma_n^\star p_{n-1}^\star\mc{O}(\delta_{n-1}).
\end{equation*}
\end{lem}

\begin{proof} This is a consequence of the geometric considerations in Section 2 of \cite{EGL01}. We introduce the following notation: Set
\begin{eqnarray*}
\psi_X&:=&\psi_n\times\id_X\\
\phi_n&:=&p_{n-1}\circ\sigma_n\\
\phi_X&:=&\phi_n\times \id_X\\
j&:=&\id_{\nt{X}}\times(q_{n-1}\circ\sigma_n)
\end{eqnarray*}
These maps fit into the following diagram, where $\pi$ is the first projection:
\[\xymatrix{\nt{X}\ar[r]^(.4)j&\nt{X}\times X\ar[ld]^(.45){\psi_X}\ar[rd]_(.45){\phi_X}\ar[r]^(.55)\pi&\nt{X}\\
\tn{X}\times X&&\tm{X}\times X}.\]
Using this notation, we can write sequence (6) from \cite{EGL01} as follows:
\begin{equation*}
0\rightarrow j_\star\mc{O}(D_n)\rightarrow \psi_X^\star\mc{O}_{\Xi_n}\rightarrow\phi_X^\star\mc{O}_{\Xi_{n-1}}\rightarrow 0.
\end{equation*}
To this sequence we want to apply $\pi_\star$. First note that $\pi\circ j=\id_{\nt{X}}$. Furthermore, we have a commutative diagram
\[\xymatrix{\nt{X}\times X\ar[r]^(.55)\pi\ar[d]^(.45){\psi_X}&\nt{X}\ar[d]^(.45){\psi_n}\\
\tn{X}\times X\ar[r]^(.55){p_n}&\tn{X},}
\]
where $p_n$ is the first projection as usual. Since the projections are flat we have
\[\pi_\star\circ\psi_X^\star\simeq\psi_n^\star\circ p_{n\star}.\]
Define
\[\tnmc{O}_X:=p_{n\star}\mc{O}_{\Xi_n}.\]
This is a rank $n$ vector bundle on $\tn{X}$ with determinant $\mc{O}(\delta_n)$. Thus we find $\pi_\star\psi_X^\star\mc{O}_{\Xi_n}\simeq\psi_n^\star\tnmc{O}_X$ and similarly $\pi_\star\phi_X^\star\mc{O}_{\Xi_n}\simeq\phi_n^\star\tmmc{O}_X$. Altogether we see that we have an exact sequence on $\nt{X}:$
\begin{equation*}\label{eqhigherntautseq}
0\rightarrow \mc{O}(D_n)\rightarrow \psi_n^\star\tnmc{O}_X \rightarrow \phi_n^\star\tmmc{O}_X \rightarrow 0.
\end{equation*}
Hence taking determinants yields the lemma.
\end{proof}

\begin{cor}\label{corhigherninduct}
We have
\begin{equation*}
\psi_n^\star\delta_n=[D_n]+\sigma_n^\star p_{n-1}^\star\delta_{n-1}.
\end{equation*}
\end{cor}

Next, there is a recursive formula for classes in $\NS \tn{X}$ coming from $X$:

\begin{lem}\label{eqhigherninductg}
For every class $l\in\NS X$ we have
\[\psi_n^\star l_{\tn{X}} = \sigma_n^\star(p_{n-1}^\star l_{\tm{X}}+q_{n-1}^\star l).\]
\end{lem}

\begin{proof} Denote the natural degree $n$ covering $S^{n-1}\times X\rightarrow S^nX$ by $f_n$. We have 
\begin{align*}
\begin{array}{lclclr}
\psi_n^\star l_{\tn{X}} &=& \psi_n^\star\rho_n^\star (l^{\boxplus n})^{\Sy{n}} &=& \sigma_n^\star(\rho_{n-1}\times \id_X)^\star f_n^\star(l^{\boxplus n})^{\Sy{n}}&\\
&=&\sigma_n^\star(\rho_{n-1}\times \id_X)^\star((l^{\boxplus n-1})^{\Sy{n-1}}\boxplus l)& =& \sigma_n^\star(\rho_{n-1}^\star(l^{\boxplus n-1})^{\Sy{n-1}}\boxplus l)&\\
&=& \sigma_n^\star(p_{n-1}^\star l_{\tm{X}}+q_{n-1}^\star l).&&&\hspace{20pt}\qedhere
\end{array}
\end{align*}
\end{proof}

\begin{rem}\label{remgeomhilbn}
We leave it to the reader to formulate and prove the analogous result to the lemma above for line bundles instead of cohomology classes.
\end{rem}

If $X$ is regular (i.\ e.\ $h^1(X,\mc{O}_X)=0$), we have
\[\NS\tn{X}\cong \NS X \oplus \ZZ\delta_n.\]

Finally, let us also consider the case that our surface is an abelian surface $A$. We have the group law $s_n\colon A^n\rightarrow A$ which factors through $\tilde{s}_n\colon S^nA\rightarrow A$ and we denote the composition $\tilde{s}_n\circ\rho_n$ by $m_n$:
\[\xymatrix{
\tn{A}\ar[dr]_{m_n}\ar[r]^{\rho_n}&S^nA\ar[d]^(.4){\tilde{s}_n}&A^n\ar[l]\ar[dl]^{s_n}\\
&A
}\]
Thus we have embeddings
\begin{eqnarray}
(-)_{M_n}\colon\NS A\hookrightarrow \NS A^n,\quad l\mapsto s_n^\star l-l^{\boxplus n}\quad\text{and}\label{eqgeomhilbnmn}\\[4pt]
(-)_{m_n}\colon\NS A\hookrightarrow\NS\tn{A},\quad l\mapsto m_n^\star l-l_{\tn{X}}.\notag
\end{eqnarray}

\begin{rem}\label{lemgeomhilb2linind}
As one can easily see from the K\"unneth decomposition, the class $l_{M_n}$ is linearly independent of the summand $(\NS A)^{\boxplus n}$ inside $\NS A^n.$
\end{rem}

Finally, let us briefly introduce the generalised Kummer varieties. If one mimics the construction of Hilbert schemes to the case of abelian surfaces, one again obtains Ricci flat manifolds. But they are not simply connected and contain additional factors in the Beauville$-$Bogomolov decomposition. To get rid of these factors we consider (for an abelian surface $A$) the fibre
\[K_n(A):=m_n^{-1}(0)\]
and call it \em generalised Kummer variety. \em It is a $(2n-2)$-dimensional irreducible symplectic manifold (cf.\ \cite{Beau83}). In the case $n=2$ this just gives the Kummer surface $\Km{A}$. 

Note that some authors (e.g.\ \cite{Beau83}) use the notation $K_n$ for the generalised Kummer variety of dimension $2n$. We will use the notation introduced above which is also used in \cite{Huy99}.

\subsection{Tautological Sheaves}\label{subsectiontaut}
Let us give the definition of tautological sheaves, the objects of main interest in this article. Fix a sheaf $\mc{F}$ on $X$ and recall that there is the universal subscheme $\Xi_n\subset X\times\tn{X}.$ Furthermore we have the two projections $p_n\colon X\times\tn{X}\rightarrow \tn{X}$ and $q_n\colon X\times\tn{X}\rightarrow X$.

\begin{defi}
The \em tautological sheaf associated with $\mc{F}$ \em is defined as
\[\tnmc{F}:=p_{n\star}(q_n^\star\mc{F}\otimes \mc{O}_{\Xi_n}).\]
\end{defi}

\begin{rem}\label{remtautdefres}
Very important for the study of tautological sheaves is the following observation: The universal subscheme $\Xi_n$ and the nested Hilbert scheme $\nt{X}$ are isomorphic outside codimension four subschemes (cf.\ Section \ref{subsectiongeom}). Let $U$ denote the open subset where they are actually isomorphic. The restrictions of $q_n^\star\mc{F}$ and $\sigma_n^\star q_{n-1}^\star\mc{F}$ to $U$ are naturally isomorphic. Thus the restriction of $\tnmc{F}$ to the image $p_n(U)$ in $\tn{X}$ is isomorphic to $\widetilde{\tnmc{F}}:=\psi_{n\star}\sigma_n^\star q_{n-1}^\star\mc{F}$ (restricted to $\psi_n(U)=p_n(U)$). Hence we can use $\widetilde{\tnmc{F}}$ instead of $\tnmc{F}$ as long as we want to study properties that are not sensible with respect to modifications in codimension four. In the case $n=2$ we, in fact, have $\widetilde{\tamc{F}}\simeq\tamc{F}$.
\end{rem}

The restriction of $p_n$ to $\Xi_n$ is a flat covering of degree $n$. Hence the following lemma:

\begin{lem}
If $\mc{F}$ is locally free (torsion-free, resp.), so is $\tnmc{F}.$ If $\mc{F}$ has rank $r,$ then $\tnmc{F}$ has rank $nr.$
\end{lem}

\begin{proof} \cite[Rem.\ 2.5 and Lem.\ 2.23]{Scal09b}.\end{proof}

\begin{lem}\label{lemtautdual}
Let $\mc{F}$ be a locally free sheaf on $X.$ Then
\begin{equation*}
(\tnmc{F})^\vee \simeq \tn{(\mc{F}^\vee)}\otimes \mc{O}(\delta_n).
\end{equation*}
\end{lem}

\begin{proof} Recall that $p_n^\star\mc{O}(\delta_n)$ is the relative canonical sheaf of the degree $n$ covering $p_n$. Using Grothendieck$-$Verdier duality, we have
\[\begin{array}{lclclr}
(\tnmc{F})^\vee &=& \mc{H}om_{\mc{O}_{\tn{X}}}(p_{n\star}q_n^\star\mc{F}, \mc{O}_{\tn{X}})  &\simeq &
p_{n\star}\mc{H}om_{\mc{O}_{\Xi_n}}(q_n^\star\mc{F},p_n^\star\mc{O}(\delta_n))\\[7pt]
&\simeq & p_{n\star}(q_n^\star\mc{F}^\vee\otimes p_n^\star\mc{O}(\delta_n))
&\simeq&\tn{(\mc{F}^\vee)}\otimes \mc{O}(\delta_n).&\qedhere
\end{array}\]
\end{proof}

\begin{lem}\label{lemtautc1}
We have the following formula for the first Chern class of $\tnmc{F}$:
\[c_1(\tnmc{F})=c_1(\mc{F})_{\tn{X}}-\rk(\mc{F})\delta_n.\]
\end{lem}

\begin{proof} The map $p_n\colon \Xi_n\rightarrow \tn{X}$ is a flat covering of degree $n$ with branch divisor $\delta_n$. Thus the class of the relative canonical bundle of $p_n$ is $p_n^\star\delta_n$. Hence the Grothendieck$-$Riemann$-$Roch theorem reads
\begin{eqnarray*}
\begin{array}{lclcl}
\ch(\tnmc{F})&=&\ch(p_{n\star}q_n^\star\mc{F})&=&p_{n\star}(q_n^\star\ch(\mc{F})\cdot\td_{p_n})\\[4pt]
&&&=&
p_{n\star}((\rk(\mc{F}),q_n^\star c_1(\mc{F}),\dots)(1,-\frac{1}{2}p_n^\star\delta_n,\dots))\\[4pt]
&&&=&p_{n\star}(\rk(\mc{F}),q_n^\star c_1(\mc{F})-\frac{1}{2}\rk(\mc{F})p_n^\star\delta_n,\dots)\\[4pt]
&&&=&(n\rk(\mc{F}),p_{n\star}q_n^\star c_1(\mc{F})-\rk(\mc{F})\delta_n,\dots).
\end{array}
\end{eqnarray*}
Note that --- as in the $n=2$ case --- we have $p_{n\star}p_n^\star\delta_n=2\delta_n$ because along the divisor $2\delta_n$ two sheets of the degree $n$ covering come together.

Certainly the first Chern class is independent of modifications in codimenion four, i.e.\ $p_{n\star}q_n^\star c_1(\mc{F})=\psi_{n\star}\sigma_n^\star q_{n-1}^\star c_1(\mc{F}).$ Denote by $f_n\colon S^{n-1}X\times X\rightarrow S^nX$ the degree $n$ covering and by $\mathrm{pr}_2\colon S^{n-1}X\times X\rightarrow X$ the second projection. We have
$\psi_{n\star}\sigma_n^\star q_{n-1}^\star c_1(\mc{F})=\rho_n^\star f_{n\star}\mathrm{pr}_2^\star c_1(\mc{F})=c_1(\mc{F})_{\tn{X}}.$ \end{proof}

Next, we want to summarise the results of Scala and Krug about global sections and extensions of tautological sheaves. These formulas turn out to be a powerful tool to analyse stability and deformations of these sheaves.

\begin{thm}\label{thmtautcohextcoh}
For every sheaf $\mc{F}$ and every line bundle $\mc{L}$ on $X$ we have
\begin{equation*}\label{eqtautcoh}
\cH^\ast(\tn{X},\tnmc{F}\otimes\mc{L}_{\tn{X}})\cong \cH^\ast(X,\mc{F}\otimes\mc{L})\otimes S^{n-1} \cH^\ast(X,\mc{L}).
\end{equation*}
\end{thm}

\begin{proof} \cite[Cor.\ 4.5]{Scal09b}, \cite[Thm.\ 6.17]{Kru11}.\end{proof}

We continue by stating Krug's formula for the extension groups of tautological sheaves:

\begin{thm}
Let $\mc{F}$ and $\mc{E}$ be sheaves and $\mc{L}$ and $\mc{M}$ be line bundles on $X.$ We have
\begin{eqnarray}\label{eqtautextg}
\Ext_{\tn{X}}^\ast(\tnmc{E}\otimes\mc{L}_{\tn{X}},\tnmc{F}\otimes\mc{M}_{\tn{X}}) \hspace{5pt}\cong
\begin{array}{c}\\
\Ext_X^\ast(\mc{E}\otimes\mc{L},\mc{F}\otimes\mc{M})\otimes S^{n-1}\Ext_X^\ast(\mc{L},\mc{M})\\\bigoplus\\
\Ext_X^\ast(\mc{E}\otimes\mc{L},\mc{M})\otimes\Ext_X^\ast(\mc{L},\mc{F}\otimes\mc{M})\otimes\\ S^{n-2}\Ext_X^\ast(\mc{L},\mc{M}).
\end{array}
\end{eqnarray}
\end{thm}

\begin{proof} \cite[Thm.\ 6.17]{Kru11}.\end{proof}

Krug also gave a description how to compute Yoneda products on these extension groups (cf.\ \cite[Sect.\ 7]{Kru11}). The general formulas are extremely long. We will give a more detailed account on them as needed.

Let us finish this section by deriving a special case of formula (\ref{eqtautextg}).

\begin{cor}\label{cortautext}
Let $X$ be a $K3$ surface and let $\mc{F}$ be a sheaf on $X$ satisfying $h^2(\mc{F})=0.$ Then we have
\begin{eqnarray}
\Hom_{\tn{X}}(\tamc{F},\tamc{F})&\cong&\Hom_X(\mc{F},\mc{F}),\notag\\[5pt]
\Ext_{\tn{X}}^1(\tamc{F},\tamc{F}) &\cong& \Ext_X^1(\mc{F},\mc{F})\bigoplus  \cH^0(X,\mc{F})\otimes \cH^1(X,\mc{F})^\vee. \label{eqtautext}
\end{eqnarray}
\end{cor}

\begin{rem}
From these equations we can deduce that tautological sheaves $\tamc{F}$ associated with stable sheaves $\mc{F}\not\simeq\mc{O}_X$ are always simple: By Serre Duality a stable sheaf $\mc{F}\not\simeq\mc{O}_X$ on a $K3$ surface satisfies either $h^2(\mc{F})=0$ or $h^0(\mc{F})=0$ and by twisting with a suitable line bundle we may assume that $h^2(\mc{F})=0$. This is a first indication that tautological sheaves might be stable.
\end{rem}

\subsection{Polarisations and Slopes}
In this section we shall talk about polarisations on the Hilbert scheme of points on a surface. In general the ample cone of these varieties is not completely known. Nevertheless, if we fix a polarisation $H$ on our surface $X$, we will define polarisations $H_N$ on $\tn{X}$, depending on $H$ and an integer $N$.Furthermore, we shall derive and discuss the slopes of tautological sheaves with respect to these polarisations. This will be important when we want to study the stability of these sheaves in Chapters \ref{sectionhighern} and \ref{sectionabsurf}.\\

Fix a smooth projective surface $X$ and an ample class $H\in\NS X$. For any integer $N$ we consider the class
\[H_N:=NH_{\tn{X}}-\delta_n\in\NS\tn{X}.\]

\begin{lem}
For all sufficiently large $N,$ the class $H_N$ is ample.
\end{lem}

\begin{proof}
For $n=2$ the Hilbert$-$Chow morphism $\rho$ is a blow up. The class $H_{\ta{X}}$ is the pullback of an ample class on $S^2X$ and $-\delta$ is ample on the fibres of $\rho$. For $n>2$ we proceed by induction. By Corollary \ref{corhigherninduct} and Lemma \ref{eqhigherninductg} we have
\[\psi_n^\star H_N=\sigma_n^\star (p_{n-1}^\star (NH_{\tm{X}}-\delta_{n-1})+Nq_{n-1}^\star H) -[D_n]. \]
By induction $NH_{\tm{X}}-\delta_{n-1}$ is ample on $\tm{X}$. Hence for sufficiently large $N$, $\psi_n^\star H_N$ is ample. By \cite[Cor. 1.2.24]{Laz04} $H_N$ is ample, too.
\end{proof}

Thus we have a natural candidate for a polarisation of the Hilbert scheme and, as it turns out, in many cases tautological sheaves are stabe with respect to these polarisations. We will briefly summarise the results obtained in \cite{Wan12} on the stability of tautological sheaves on regular surfaces.

\begin{thm} \label{thmregsurflb}
Let $\mc{F}$ be a $\mu_H$-stable vector bundle on $X$ and assume $\mc{F}\not\simeq\mc{O}_X.$ Then for sufficiently large $N,$ the tautological vector bundle $\tamc{F}$ on $\ta{X}$ does not contain any $\mu_{H_N}$-destabilising line subbundles.
\end{thm}

\begin{thm}\label{thmregsurfcasesr1}
Let $\mc{F}$ be a rank one torsion-free sheaf on $X$ satisfying $\det\mc{F}\not\simeq\mc{O}_X.$ Then for $N$ sufficiently large, $\tamc{F}$ is a $\mu_{H_N}$-stable rank two torsion-free sheaf on $\ta{X}.$
\end{thm}

\begin{thm}\label{thmregsurfcasesr2}
Let $\mc{F}$ be a rank two $\mu_H$-stable sheaf on $X$ and assume $\det\mc{F}\not\simeq\mc{O}_X.$ Then for $N$ sufficiently large, $\tamc{F}$ is a $\mu_{H_N}$-stable rank four sheaf on $\ta{X}.$
\end{thm}

Next, we want to compute slopes of tautological sheaves also in the case $n>2.$ Hence we need to compute intersection numbers. We have the following general result:

\begin{lem}\label{lempolaintersect}
Let $l$ be a class in $\NS X.$ We have 
\begin{eqnarray}
l_{\tn{X}}.H_{\tn{X}}^{2n-1}&=&\frac{n}{2^{n-1}}(l.H)(H^2)^{n-1}\text{ and} \label{eqpolaint1}\\[7pt]
\delta_n.H_{\tn{X}}^{2n-1}&=&0,\label{eqpolaint2}
\end{eqnarray}
where on the right hand side of \em (\ref{eqpolaint1}) \em we consider the intersection in $\NS X.$
\end{lem}

\begin{proof} Both $l_{\tn{X}}$ and $H_{\tn{X}}$ are pullbacks from $S^nX$ along the Hilbert$-$Chow morphism. We pull back along the $n!$-fold covering $X^n\rightarrow S^nX$ and obtain the classes $l^{\boxplus n}$ and $H^{\boxplus n}$, respectively. We have
\begin{eqnarray*} 
l_{\tn{X}}.H_{\tn{X}}^{2n-1}=\frac{1}{n!}(l^{\boxplus n})(H^{\boxplus n})^{2n-1}=\frac{1}{n!}\binom{n}{1,2,\dots,2}n(l.H)(H^2)^{n-1}=\frac{n}{2^{n-1}}(l.H)(H^2)^{n-1}.
\end{eqnarray*}
In order to prove (\ref{eqpolaint2}), it is certainly enough to show that
\begin{equation}
(\psi_n^\star\delta_n).(\psi_n^\star H_{\tn{X}}^{2n-1}) =0.\label{eqpolapullback}
\end{equation}
We will use an induction argument. For $n=2,$ equation (\ref{eqpolapullback}) reads
\[D.\sigma^\star (H^{\boxplus 2})^3=0.\]
This is true by Lemma 1.1c) in \cite{Wan12}. Now for the induction step we use Lemmata \ref{eqhigherninductd} and \ref{eqhigherninductg}:
\begin{eqnarray*}
\begin{array}{clcl}
&(\psi_n^\star\delta_n).(\psi_n^\star H_{\tn{X}}^{2n-1}) && \\[7pt]
=&\sigma_n^\star{\big(}(p_{n-1}^\star H_{\tm{X}}+q_{n-1}^\star H)^{2n-1}p_{n-1}^\star\delta_{n-1}{\big)} &+& \underbrace{\sigma_n^\star(p_{n-1}^\star H_{\tm{X}}+q_{n-1}^\star H)^{2n-1}[D_n]}_{||}\\[8pt]
=&\binom{2n-1}{2}p_{n-1}^\star (\delta_{n-1}.H^{2n-3})q_{n-1}^\star H^2 &+& \textrm{(\hspace{55pt} \em as above \em \hspace{55pt})}.\end{array}
\end{eqnarray*}
Now by induction the first term vanishes. And for the second term we can apply exactly the same reasoning as in \cite[Lemma 1.1c)]{Wan12}.\end{proof}

\begin{cor}\label{corpolaslope}
Let $\mc{L}$ be a line bundle on $X$ with first Chern class $l$ and $\mc{F}$ a sheaf of rank $r$ and first Chern class $f.$ We have the following expansions for the slopes of $\tnmc{F}$ and $\mc{L}$ with respect to $H_N$:
\begin{eqnarray*}
\mu_{H_N}(\mc{L}_{\tn{X}})&=& N^{2n-1}\frac{n}{2^{n-1}}(l.H)(H^2)^{n-1}+O(N^{2n-2})\quad\text{and}\\[7pt]
\mu_{H_N}(\tnmc{F})&=& N^{2n-1}\frac{n}{2^{n-1}}\frac{1}{nr}(f.H)(H^2)^{n-1}+O(N^{2n-2}).
\end{eqnarray*}
\end{cor}

If $X=A$ is an abelian surface, there is another candidate for a polarisation. Recall that we have a summation morphism $m_n\colon\tn{A}\rightarrow A$ and at least one additional summand in $\NS\tn{A}$ containing $(\NS A)_{M_n}.$ As will be explained in Lemma \ref{lemabsurfslopedelta}, classes in $(\NS A)_{M_n}$ have degree zero with repect to the polarisation $H_N$, which turns out to be inconvenient for the proof of stability of tautological sheaves. To circumvent this issue we will consider the following polarisation:

\begin{lem}\label{lempolaabsurfslope}
For all $N\gg0$ the class
\[H^m_N:=NH_{\tn{X}}-\delta_n+Nm_n^\star H\]
is ample.
\end{lem}

\begin{proof} By \cite[Exa.\ 1.4.4]{Laz04} the pullback $m_n^\star H$ is nef. Hence we can apply \cite[Cor.\ 1.4.10]{Laz04} to conclude that $H^m_N$ is ample.\end{proof}

\section{Higher $n$}\label{sectionhighern}
In this chapter we try to generalise the results on destabilising line subbundles in \cite[Sect.\ 3]{Wan12} to higher $n$. From this generalisation we will be able to prove the stability of rank three tautological sheaves on $\tb{X}$. In this chapter we fix a polarised regular surface $(X,H).$\\

Let $\mc{F}$ be a torsion-free $\mu_H$-stable sheaf on $X$. Denote its rank by $r$ and its first Chern class by $f$. We want to show that the associated tautological sheaf $\tnmc{F}$ on $\tn{X}$ has no destabilising subsheaves of rank one. We will first assume that $\mc{F}$ is reflexive, i.e.\ locally free. Thus we may assume that a destabilising rank one subsheaf of $\tnmc{F}$ is also reflexive, that is, a line bundle.

\begin{prop}\label{propregsurfhighernlb}
For sufficiently large $N,$ there are no $\mu_{H_N}$-destabilising line subbundles in $\tnmc{F}$ of the form $\mc{L}_{\tn{X}},$ ($\mc{L}\in\Pic X$), except the case $r=1$ and $\mc{L}\simeq\mc{F}\simeq\mc{O}_X.$
\end{prop}

\begin{proof} Denote the first Chern class of $\mc{L}$ by $l$. Using Scala's calculations of cohomology groups of tautological sheaves with twists as stated in Theorem \ref{thmtautcohextcoh} we can immediately deduce the following formula for homomorphisms from line bundles of the form $\mc{L}_{\tn{X}}$ to tautological sheaves $\tnmc{F}$:
\begin{eqnarray*}
\Hom_{\tn{X}}(\mc{L}_{\tn{X}},\tnmc{F})\cong \Hom_X(\mc{L},\mc{F})\otimes\Hom_X(\mc{L},\mc{O}_X).
\end{eqnarray*}
Let us first assume $r>1$. Since $\mc{F}$ is $\mu_H$-stable, we have necessary conditions for the existence of a line subbundle of $\tnmc{F}$:
\begin{eqnarray}\label{ineqregsurfhighern}
l.H<\frac{f.H}{r}\text{ and }l.H\leq0.
\end{eqnarray}
The first inequality is due to the stability of $\mc{F}$ and the second comes from the fact that if a line bundle has a section, its first Chern class has non-negative intersection with any ample class $H.$ If $\mc{L}_{\tn{X}}\subset\tnmc{F}$ is destabilising, by Corollary \ref{corpolaslope} we must have \[l.H\geq\frac{f.H}{nr}.\]
But this is certainly a contradiction to (\ref{ineqregsurfhighern}).

If $r=1$, we can proceed as above but additionally have to consider the special case $\mc{L}\simeq\mc{F}$, i.e.\ $l.H=f.H$. The destabilising condition together with $l.H\leq0$ immediately yields $l.H=0$. But now $\Hom_X(\mc{L},\mc{O}_X)$ can only be nontrivial if $\mc{L}\simeq\mc{O}_X$.\end{proof}

Similarly to the proof of Proposition 4.1 in \cite{Wan12} we can deduce the following lemma:

\begin{lem}\label{lemregsurfhigherns}
For all locally free sheaves $\mc{G}$ and $\mc{H}$ on $X\times\tm{X}$ and all $a\in\ZZ$ we have:
\begin{eqnarray*}
\Hom_{\nt{X}}(\sigma_n^\star\mc{G}\otimes\mc{O}(aD_n),\sigma_n^\star\mc{H})&\subseteq&
\Hom_{X\times\tm{X}}(\mc{G},\mc{H})
\end{eqnarray*}
\end{lem}

Now we consider arbitrary line subbundles $\mc{L}_{\tn{X}}\otimes\mc{O}(a\delta_n),$ $\mc{L}\in\Pic X,$ $a\in\ZZ,$ and show that we can reduce to the case of Proposition \ref{propregsurfhighernlb}:

\begin{lem}\label{lemregsurfhighernd}
Let $\mc{L}_{\tn{X}}\otimes\mc{O}(a\delta_n)$ be a line bundle on $\tn{X},$ Then for any locally free sheaf $\mc{F}$ on $X$ we have
\begin{eqnarray*}
\Hom_{\tn{X}}(\mc{L}_{\tn{X}}\otimes\mc{O}(a\delta_n),\tnmc{F})&\subseteq& \Hom_{\tn{X}}(\mc{L}_{\tn{X}},\tnmc{F}).
\end{eqnarray*}
\end{lem}

\begin{proof} We use Remark \ref{remtautdefres}, adjunction, the recursive formulas in Corollary \ref{corhigherninduct} and Lemma \ref{eqhigherninductg}, Remark \ref{remgeomhilbn}, Lemma \ref{lemregsurfhigherns} above and, finally, the K\"unneth formula:
\begin{eqnarray*}
\begin{array}{cl}
&\Hom_{\tn{X}}(\mc{L}_{\tn{X}}\otimes\mc{O}(a\delta_n),\tnmc{F})\\[4pt]
\cong & \Hom_{\tn{X}}(\mc{L}_{\tn{X}}\otimes\mc{O}(a\delta_n),\psi_{n\star}\sigma_n^\star q_{n-1}^\star\mc{F})\\[4pt]
\cong & \Hom_{\nt{X}}(\psi_n^\star(\mc{L}_{\tn{X}}\otimes\mc{O}(a\delta_n)),\sigma_n^\star q_{n-1}^\star\mc{F})\\[4pt]
\cong & \Hom_{\nt{X}}\big{(}\sigma_n^\star\big{(}p_{n-1}^\star (\mc{L}_{\tm{X}}\otimes\mc{O}(a\delta_{n-1}))\otimes q_{n-1}^\star \mc{L}\big{)}\otimes\mc{O}(aD_n),\sigma_n^\star q_{n-1}^\star\mc{F}\big{)}\\[4pt]
\subseteq & \Hom_{X\times\tm{X}}(p_{n-1}^\star (\mc{L}_{\tm{X}}\otimes\mc{O}(a\delta_{n-1}))\otimes q_{n-1}^\star \mc{L},q_{n-1}^\star\mc{F})\\[4pt]
\cong &\Hom_X(\mc{L},\mc{F})\otimes\Hom_{\tm{X}}(\mc{L}_{\tm{X}}\otimes\mc{O}(a\delta_{n-1}),\mc{O}_{\tm{X}}).
\end{array}
\end{eqnarray*}
Now we can use induction to conclude. The initial step for $n=2$ is, again, settled by similar arguments as in \cite[Prop.\ 4.1]{Wan12}.\end{proof}

We are ready to prove the first main result of this section.

\begin{prop}\label{propregsurflb}
Let $\mc{F}$ be a torsion-free $\mu_H$-stable sheaf on $X.$ Assume that its reflexive hull $\mc{F}^{\vee\vee}\not\simeq\mc{O}_X.$ Then $\tnmc{F}$ does not contain $\mu_{H_N}$-destabilising subsheaves of rank one for all $N\gg0.$
\end{prop}

\begin{proof} If $\mc{F}$ is locally free, we can simply apply Proposition \ref{propregsurfhighernlb} and Lemma \ref{lemregsurfhighernd} above. If $\mc{F}$ is not locally free, we proceed as usual in order to reduce to the locally free case:

Let $\mc{E}:=\mc{F}^{\vee\vee}$ be the reflexive hull of $\mc{F}$. It has the same rank and first Chern class and is a locally free $\mu_H$-stable sheaf. Thus we get an injection of $\tnmc{F}$ into the locally free tautological sheaf $\tnmc{E}$ which again has the same rank and first Chern class. Now we can apply the lemmata above. Note that every destabilising subsheaf of $\tnmc{F}$ also destabilises $\tnmc{E}$. \end{proof}

Since the tautological sheaf on $\tb{X}$ associated with a rank one sheaf has rank three, the above proposition is enough to show that these sheaves are stable (except $\tbmc{O}_X$, of course).

\begin{thm}\label{thmregsurf3}
Let $\mc{F}$ be a torsion-free rank one sheaf on $X$ satisfying $\det\mc{F}\not\simeq\mc{O}_X.$ Then for all sufficiently large $N$ the associated rank three sheaf $\tbmc{F}$ on $\tb{X}$ is $\mu_{H_N}$-stable.
\end{thm}

\begin{proof} As usual we can reduce to the case that $\mc{F}$ is locally free. We have seen that $\tbmc{F}$ cannot contain destabilising subsheaves of rank one. But any destabilising subsheaf of rank two yields a rank one destabilising subsheaf of the dual sheaf. Using Lemma \ref{lemtautdual} we are done.\end{proof}

\section{Abelian Surfaces and Generalised Kummer Varieties}\label{sectionabsurf}
In this Chapter we study the stability of tautological sheaves on Hilbert schemes of points on abelian surfaces and their restrictions to the associated generalised Kummer varieties.

\subsection{Geometric Considerations: $n=2$}\label{subsectionabsurfgeom2}
As explained already at the end of Section \ref{subsectiongeom}, for abelian surfaces $A$ the structure of the N\'{e}ron$-$Severi group of $\NS \tn{A}$ is more complicated as in the case of regular surfaces. In order to prove the stability of tautological sheaves nevertheless, we will restrict to the case of principally polarised abelian surfaces (p.p.a.s.) $(A,H)$ of Picard rank one. We begin with a technical lemma:

\begin{lem}\label{lemabsurfhom}
Let $A$ and $A'$ be complex tori. Then we have an isomorphism of abelian groups
\[\NS(A\times A')\cong \NS(A)\oplus\Hom(A',\widehat{A})\oplus \NS(A'),\]
where $\widehat{A}$ denotes the dual torus.
\end{lem}

\begin{proof} The proof of this lemma was pointed out to me by H.\ Ohashi. The K\"unneth formula yields a decomposition of $\NS(A\times A')$ into direct summands, two of which are naturally isomorphic to $\NS(A)$ and $\NS(A')$, respectively. The remaining summand can be written as
\begin{equation}\label{eqnsproduct}
\big{(}(\cH^{1,0}(A)\otimes \cH^{0,1}(A'))\oplus(\cH^{0,1}(A)\otimes \cH^{1,0}(A'))\big{)}\cap\big{(}\cH^1(A,\ZZ)\otimes \cH^1(A',\ZZ)\big{)}.
\end{equation}
We can interpret $\cH^{1,0}(A)\otimes \cH^{0,1}(A')$ as $\Hom (\cH^{0,1}(\widehat{A}),\cH^{0,1}(A'))$ and so we see that (\ref{eqnsproduct}) is just the set of morphisms of integral Hodge structures
\[\cH^1(\widehat{A},\ZZ)\rightarrow \cH^1(A',\ZZ).\qedhere\]
\end{proof}

Denote by $s\colon A\times A \rightarrow A$ the group law. An important role will play the class $H_M:=s^\star H-H^{\boxplus2},$ which is called the \em Mumford class \em associated with $H.$ By Lemma \ref{lemgeomhilb2linind} it is linearly independent from the summand $(\NS A)^{\boxplus2}$ inside $\NS (A\times A).$ We will restrict to a certain (large) set of p.p.a.s.:
\begin{eqnarray*}
A\text{ is a p.p.a.s.\ such that}\hspace{30pt}&\left(
\begin{array}{c}\NS A\cong\ZZ H\\
\text{and}\\
\NS(A\times A)\cong\NS A^{\boxplus2}\oplus\ZZ H_M.
\end{array}\right)&(\star)
\end{eqnarray*}

\begin{lem}
The class of abelian surfaces satisfying condition $(\star)$ is the complement of a countable union of analytic subsets in the moduli space $\mc{A}_2$ of principally polarised abelian surfaces.
\end{lem}

\begin{proof}It is well known that the set of abelian surfaces of Picard rank one is very general in $\mc{A}_2$ (cf.\ \cite[Exerc.\ 1a), p.\ 244]{BL92}). So we can assume $\NS A\cong\ZZ H$, where $H$ is a principal polarisation. For the second assertion note that we always have $\NS A^{\boxplus2}\subset\NS(A\times A)$ and there is at least one additional summand containing the class $H_M$. In order to show that for a very general abelian surface no more summands show up, we use Theorem 9.1 in \cite{BL92} (or Theorem 4.7.1 in \cite{Ara}) --- both stating that a very general $A$ satisfies $\Hom(A,A)\cong\ZZ$ --- and Lemma \ref{lemabsurfhom} above. Finally, let us show that the class $H_M$ is primitive. By Lemma \ref{lemabsurfhom} it is certainly enough to prove that $H_M$ corresponds to the identity in $\Hom(A,A).$ To the class $H$ we can associate the homomorphism \[\varphi_H\colon A\rightarrow\widehat{A},\quad x\mapsto t_x^\star\mc{O}(H)\otimes\mc{O}(-H),\]
which is an isomorphism since $H$ is a principal polarisation and it yields an identification $\Hom(A,A)\cong\Hom(A,\widehat{A}).$ Furthermore, on $A\times \widehat{A}$ we have the Poincar\'{e} line bundle $\mc{P}.$ Altogether we can describe the inclusion $\Hom(A,A)\hookrightarrow\NS(A\times A)$ as follows:
\[\begin{array}{ccccc}
\Hom(A,A)&\rightarrow&\Hom(A,\widehat{A})&\rightarrow&\NS(A\times A)\\
f&\mapsto& \varphi_H\circ f,\\
&&\varphi&\mapsto& c_1((\id_A\times\varphi)^\star\mc{P}).
\end{array}\]
Finally, equation $(9.8)$ in \cite[Chapt.\ 9]{Huy06} says that \[(\id_A\times\varphi_H)^\star\mc{P}\simeq s^\star\mc{O}(H)\otimes\mc{O}(-H)^{\boxtimes2}.\qedhere\]
\end{proof}

\begin{rem}\label{remabsurfgeom}
By the considerations on Page 198 in \cite{Huy06} we see that the first Chern class of $(\id_A\times\varphi_H)^\star\mc{P}$ is contained in the K\"unneth summand
$H^1(A,\ZZ)^{\boxtimes2}.$ Thus the identity
\[s^\star H=H^{\boxplus2}+H_M\]
is exactly the K\"unneth decomposition.

If $H$ is not a principal polarisation, the homomorphism $\varphi_H$ is no longer an isomorphism. Thus it is not clear if the class $H_M$ is primitive.
\end{rem}

\begin{cor}
If $A$ satisfies $(\star),$ we have
\begin{eqnarray*}
\NS\ta{A}\cong \ZZ H_{\ta{A}}\oplus\ZZ H_m\oplus\ZZ\delta.
\end{eqnarray*}
\end{cor}
\begin{proof}
For details of the geometry of the Hilbert schemes of two points on a surface we refer to Section 1 in \cite{Wan12}. The Corollary follows easily from the assumption $(\star)$ and the fact that $\widetilde{A\times A}\xrightarrow{\psi}\ta{A}$ is the $\Sy{2}$-quotient.
\end{proof}

We continue by deriving intersection numbers and slopes for the case $n=2.$

\begin{lem}\label{lemabsurfident2}
We have the following identities in $\cH^8(A\times A,\ZZ)\cong\ZZ$:
\begin{eqnarray}
s^\star H^2\cdot (H\otimes H)&=&4,\label{eqabsurfident}\\[4pt]
s^\star H^2\cdot (1\otimes H^2)=s^\star H^2\cdot (H^2\otimes 1)&=&4\quad\text{and}\notag\\[4pt]
s^\star H\cdot (H\otimes H^2)=s^\star H\cdot (H^2\otimes H)&=&4.\notag
\end{eqnarray}
\end{lem}

\begin{proof} The first equality is true for all p.p.a.s.\ $A$ if and only if it is true for one. Thus we may choose $A$ to be the product of two elliptic curves $E$ and $E'$ and we represent $H$ by $E\times\{0\}+\{0\}\times E'$. We shall prove the lemma by replacing this representation of the class $H$ by appropiate translates such that the intersection (\ref{eqabsurfident}) becomes transversal and then calculate the set-theoretic intersection. Thus let $x_1,x_2$ and $y_1,y_2$ be points on $E$ and $E',$ respectively. For all but a finite number of choices of these four points, the following intersection in $A\times A=(E\times E')\times (E\times E')$ is transversal:
\begin{align*}
\begin{array}{cl}
&s^{-1}(H+(x_1,y_1))\cap s^{-1}(H+(x_2,y_2))\cap\pi_1^{-1}H\cap\pi_2^{-1}H\\
=&\big{\{}(0,y_2,x_1,0),(0,y_1,x_2,0),(x_1,0,0,y_2),(x_2,0,0,y_1)\,\big{\}}.
\end{array}
\end{align*}
The other equalities can be proven in the same way.\end{proof}

\begin{cor}\label{corabsurfident}
We have
\begin{eqnarray}
H_M(H^2\otimes H)=H_M(H\otimes H^2)&=&0 \quad\text{ and}\label{eqabsurf0}\\[4pt]
(H_M)^2\cdot H\otimes H &=&-4.\notag
\end{eqnarray}
\end{cor}

\begin{proof}
The first equality follows directly from the lemma above. Furthermore, we have
\[(H_M)^2=s^\star H^2-2s^\star H\cdot H^{\boxplus2}+H^2\otimes1+1\otimes H^2+2H\otimes H.\]
Intersecting with $H\otimes H$ yields
\[(H_M)^2\cdot H\otimes H=4-2\cdot2\cdot4+0+0+2\cdot 4=-4.\qedhere\]
\end{proof}

In the case of regular surfaces we used the polarisation $H_N:=NH_{\ta{X}}-\delta.$ The following lemma indicates that this polarisation is not the ideal choice in the abelian surface case. Note that in analogy to the definition of the Mumford class we defined
\[H_m:=m^\star H-H_{\ta{A}}.\]

\begin{lem}\label{lemabsurfslopedelta}
We have the following expansion:
\begin{eqnarray*}
H_m\cdot H_N^3=0+O(N^2).
\end{eqnarray*}
\end{lem}

\begin{proof} We pullback along the double cover $\psi\colon\widetilde{X\times X}\rightarrow \ta{X}.$ Note that $\psi^\star H_{\ta{X}}=\sigma^\star H^{\boxplus2}$ and $\psi^\star H_m=\sigma^\star H_M.$ We have
\begin{eqnarray*}
\begin{array}{lclcl}
H_m\cdot H_N^3&=&\frac{1}{2}\sigma^\star H_M\cdot(\sigma^\star H_N)^3&=&\frac{1}{2}\sigma^\star(H_M\cdot(H^{\boxplus 2})^3) + O(N^2)\\[4pt]
&&&=&\frac{1}{2}H_M(3H^2\otimes H+3H\otimes H^2)+O(N^2).
\end{array}
\end{eqnarray*}
Now we use (\ref{eqabsurf0}) in Corollary \ref{corabsurfident} and we are done.\end{proof}

Thus the leading term in the expansion of the slope of a line bundle with first Chern class $H_m$ is zero.

In Lemma \ref{lempolaabsurfslope} we defined the polarisation $H^m_N:=NH_{\ta{X}}-\delta+Nm^\star H$. With respect to this polarisation the slope of line bundles with first Chern class $m^\star H$ does not vanish:

\begin{lem}\label{lemabsurf2slope}
We have the following expansions:
\begin{eqnarray*}
(H_N^m)^3H_{\ta{A}}&=& 72N^3 + O(N^2),\\[4pt]
(H_N^m)^3H_m&=& -36N^3 + O(N^2)\quad\text{and}\\[4pt]
(H_N^m)^3m^\star H&=&36N^3+O(N^2).
\end{eqnarray*}
\end{lem}

\begin{proof} First observe that by definition
\[(H_N^m)^3m^\star H=(H_N^m)^3H_{\ta{A}}+(H_N^m)^3H_m.\]
Therefore, we will only prove the first two equalities. Note that we have $H^3\otimes 1=1\otimes H^3=0=s^\star H^3.$ We write down the expansion of $\psi^\star(H_N^m)^3$:
\begin{eqnarray}
\psi^\star(H_N^m)^3 &=&N^3\psi^\star(H_{\ta{X}}+m^\star H)^3+O(N^2)\notag\\[4pt]
&=& N^3\sigma^\star\big{(}3(H^2\otimes H + H\otimes H^2) + 3s^\star H (H^2\otimes 1+2H\otimes H+1\otimes H^2) \notag\\[4pt]
&& + 3s^\star H^2 (H\otimes1 + 1\otimes H)\big{)}+\, O(N^2).\label{eqabsurfgeomhn3}
\end{eqnarray}
Using Lemma \ref{lemabsurfident2}, we have
\begin{eqnarray*}
\psi^\star((H_N^m)^3H_{\ta{A}})&=&\psi^\star(H_N^m)^3\sigma^\star H^{\boxplus2}\\[4pt]
&=&N^3\big{(}6H^2\otimes H^2+18s^\star H(H^2\otimes H)+6s^\star H^2(H^2\otimes 1)\\[4pt]
&&+\,6s^\star H^2(H\otimes H)\big{)}+O(N^2)\\[4pt]
&=&N^3(6\cdot4+18\cdot4+6\cdot4+6\cdot4)+O(N^2)\\[4pt]
&=&144N^3+O(N^2).
\end{eqnarray*}
For the second equality we use Corollary \ref{corabsurfident}: We do not have to consider the term $3(H^2\otimes H+H\otimes H^2)$ in the expansion (\ref{eqabsurfgeomhn3}). Thus we have
\begin{eqnarray*}
\psi^\star((H_N^m)^3H_m)&=&\psi^\star(H_N^m)^3\sigma^\star(s^\star H-H^{\boxplus2})\\[4pt]
&=&N^3\big{(}3s^\star H^2(H^2\otimes 1+2H\otimes H+1\otimes H^2)-18s^\star H(H\otimes H^2)-6s^\star H^2(H\otimes H)\\[4pt]
&&-\,6s^\star H^2(H^2\otimes 1)\big{)}+O(N^2)\\[4pt]
&=&N^3(3\cdot(4+2\cdot4+4)-18\cdot4-6\cdot4-6\cdot4+O(N^2)\\[4pt]
&=&-\,72N^3+O(N^2).\hspace{235pt}\qedhere
\end{eqnarray*}
\end{proof}

\subsection{Geometric Considerations: $n=3$}\label{subsectionabsurfgeom3}
Now we turn to the case $n=3$ which is not essentially different from the case $n=2$ but the calculations are somewhat more difficult. We still assume that $(A,H)$ is an abelian surface satisfying $(\star).$ 
We denote the projections $A^3\rightarrow A$ by $\pi_i,$ $i=1,2,3$ and the projections $A^3\rightarrow A^2$ by $p_{jk},$ $1\leq j<k\leq 3.$
If $A$ satisfies $(\star),$ a similar analysis as in Lemma \ref{lemabsurfhom} yields
\begin{equation}\label{eqabsurfgeomnsa3}
\NS A^3\cong \bigoplus_{i=1}^3\ZZ\pi_i^\star H\oplus \bigoplus_{1\leq j<k\leq3}\ZZ p_{jk}^\star H
\end{equation}
Note that no more complicated effects occur since the N\'{e}ron$-$Severi group is a subgroup of the second (!) cohomology. An important class in $\NS A^3$ is $H_M.$ We need the following lemma to write this class according to decomposition (\ref{eqabsurfgeomnsa3}).

\begin{lem}\label{lemabsurfgeoms3}
We have
\[s_3^\star H = -\sum_{i=1}^3\pi_i^\star H + \sum_{1\leq j<k\leq3}p_{jk}^\star s^\star H.\]
\end{lem}

\begin{proof}
Certainly $s_3^\star H$ is $\Sy{3}$-invariant. Hence we can write
\[s_3^\star H=a\sum_{i=1}^3\pi_i^\star H + b\sum_{1\leq j<k\leq3}p_{jk}^\star s^\star H,\quad a,b\in\ZZ.\]
We will intersect with different curve classes to determine $a$ and $b.$ Fix points $x_0$ and $y_0$ in $A$ and intersect with the class $l_1:=\{x_0\}\times\{y_0\}\times H.$ We have $s_3^\star H\cdot l_1=\pi_3^\star H\cdot l_1=p_{13}^\star s^\star H\cdot l_1=p_{23}^\star s^\star H\cdot l_1=H^2$ and all other intersections vanish. This yields $1=a+2b.$
Next, we intersect with the class $l_2:=\{x_0\}\times \Delta_\star H.$ We have $\pi_2^\star H\cdot l_2=\pi_2^\star H\cdot l_2=p_{12}^\star s^\star H\cdot l_2=p_{13}^\star s^\star H\cdot l_2=H^2$ but $s_3^\star H\cdot l_2$ consists of triples $(x,y,z)$ such that $x=x_0,$ $y=z\in H$ and $x_0+2z\in H$. For a general $x_0$ this gives $H^2$ multiplied by the number of two-torsion points, which is $16.$ Furthermore, we get the same number for $p_{23}^\star m^\star H\cdot l_2$ and the remaining term vanishes. Altogether we get the following system of linear equations:
\begin{eqnarray*}
1&=&a+2b\\
16&=&2a+18b,
\end{eqnarray*}
which implies $a=-1,$ $b=1.$
\end{proof}

\begin{cor}\label{corabsurfgeom3}
We have
\[H_{M_3}=\sum_{1\leq j<k\leq3} p_{jk}^\star H_M.\]
\end{cor}

\begin{proof}
We just plug in the definition of $H_M$ from (\ref{eqgeomhilbnmn}) at the end of Section \ref{subsectiongeomhilbn}. We see that $s_3^\star= \sum_{i=1}^3\pi_i^\star H + \sum_{1\leq j<k\leq3}p_{jk}^\star  H_M.$ 
\end{proof}

From (\ref{eqabsurfgeomnsa3}) and Corollary \ref{corabsurfgeom3} we deduce
\begin{eqnarray*}
\NS\tb{A}&\cong& \ZZ H_{\tb{A}}\oplus \ZZ H_{m_3}\oplus \ZZ \delta_3.
\end{eqnarray*}

Recall that we defined our polarisation $H_N^m$ as follows:
\begin{eqnarray*}
H_N^m:=N(H_{\tb{A}}+m_3^\star H)-\delta_3.
\end{eqnarray*}
In order to calculate slopes of sheaves on $\tb{A}$, we need to have an expansion of certain intersection products in terms of $N$.

\begin{lem}\label{lemabsurfgeom5}
We have:
\begin{eqnarray}\label{eqabsurf3slope}
(H_N^m)^5\cdot H_{\tb{A}}&=&3\cdot(H_N^m)^5\cdot m_3^\star H + O(N^4).
\end{eqnarray}
\end{lem}

\begin{proof} We have
\begin{eqnarray*}
(H_N^m)^5&=&N^5\big{(}H_{\tb{A}}^5+5H_{\tb{A}}^4m_3^\star H+10H_{\tb{A}}^3(m_3^\star H)^2\big{)}+O(N^4)
\end{eqnarray*}
and we therefore need to compute all terms of the form $H_{\tb{A}}^i(m_3^\star H)^j$ with $i+j=6,$ $j\leq 2.$ (Note that $(m_3^\star H)^j$ vanishes for $j>2.$) All the divisors involved are actually pullbacks along the Hilbert$-$Chow morphism $\rho$ and we can therefore compute them on $S^3A$ or, even more easily, their pullbacks on $A^3$. Since we are not interested in the exact value but only want to compare both sides of (\ref{eqabsurf3slope}), we do not care about scaling factors like the factor $6$ when we pull back to $A^3.$ To make computations easier we define the class
\[H_s:=m_3^\star H+H_{\tb{A}}.\]
The images of the classes $H_{\tb{A}}$ and $H_s$ on $A^3$ are $H^{\boxplus 3}$ by Lemma \ref{eqhigherninductg} and $\sum_{1\leq j<k\leq3}p_{jk}^\star s^\star H$ by Lemma \ref{lemabsurfgeoms3}, respectively. For degree reasons many terms vanish. We use the fact that
\[s^\star H=H^{\boxplus2}+H_M\]
is the decomposition into K\"unneth factors (cf.\ Remark \ref{remabsurfgeom}). We use Lemma \ref{lemabsurfident2} to compute the remaining terms:

\begin{eqnarray}
H_{\tb{A}}^6&=&\binom{6}{2,2,2}\cdot \pi_1^\star H^2\cdot \pi_2^\star H^2\cdot \pi_3^\star H^2=90(H^2)^3,\notag\hspace{160pt}\\[6pt]
H_{\tb{A}}^5H_s&=&3\cdot2\cdot\binom{5}{2,2,1}\cdot(p_{12}^\star s^\star H\cdot\pi_1^\star H\cdot\pi_2^\star H^2\cdot\pi_3^\star H^2)\notag\\[6pt]
&=&3\cdot2\cdot30(H^2)^3=180(H^2)^3,\notag
\end{eqnarray}
\begin{eqnarray}
H_{\tb{A}}^4H_s^2 &=&3\cdot(H^{\boxplus3})^4\cdot p_{12}^\star s^\star H^2+6\cdot(H^{\boxplus3})^4\cdot p_{12}^\star s^\star H\cdot p_{23}^\star s^\star H\notag\\[6pt]
&=&3\cdot\Big{(}2\cdot\binom{4}{2,0,2}\cdot\pi_1^\star H^2\cdot\pi_3^\star H^2+ \binom{4}{1,1,2}\cdot\pi_1^\star H\cdot\pi_2^\star H\cdot \pi_3^\star H^2\Big{)}\cdot p_{12}^\star s^\star H^2\notag\\[6pt]
&&+\,6\cdot\Big{(}2\cdot\binom{4}{1,1,2}\cdot\pi_1^\star H\cdot\pi_2^\star H\cdot\pi_3^\star H^2+\binom{4}{1,2,1}\cdot\pi_1^\star H\cdot\pi_2^\star H^2\cdot\pi_3^\star H\notag\\[6pt]
&&+\,\binom{4}{2,0,2}\cdot\pi_1^\star H^2\cdot\pi_3^\star H^2\Big{)}\cdot p_{12}^\star s^\star H\cdot p_{23}^\star s^\star H\notag\\[6pt]
&=&3\cdot\big{(}2\cdot6\cdot(H^2)^3+12\cdot(H^2)^3\big{)}+6\cdot\big{(}2\cdot12\cdot(H^2)^3+12\cdot(H^2)^3+6\cdot(H^2)^3\big{)}\notag\\[4pt]
&=&324(H^2)^3. \notag
\end{eqnarray}
Now we resubstitute $m_3^\star H=H_s-H_{\tb{A}}.$ We have
\begin{eqnarray*}
H_{\tb{A}}^5m_3^\star H&=&(180-90)(H^2)^3=90(H^2)^3\\[4pt]
H_{\tb{A}}^4(m_3^\star H)^2&=&(324-2\cdot180+90)(H^2)^3=54(H^2)^3.
\end{eqnarray*}
Altogether we have
\begin{eqnarray*} 
(H_N^m)^5\cdot H_{\tb{A}} &=& N^5\big{(}90+5\cdot90+10\cdot54 \big{)}\cdot(H^2)^3+ O(N^4)=1080\cdot(H^2)^3+O(N^4)\\[4pt]
&&\text{and}\\[4pt]
(H_N^m)^5\cdot m_3^\star H&=& N^5\big{(}90+5\cdot 54\big{)}\cdot(H^2)^3+O(N^4)=360\cdot(H^2)^3+O(N^4).\hspace{40pt}\qedhere
\end{eqnarray*}
\end{proof}

\begin{rem}
The exact values in equation (\ref{eqabsurf3slope}) in Lemma \ref{lemabsurfgeom5}  are, of course, of no interest. Much more important is the fact that the degree of $H$ and $3m_3^\star H$ with respect to $H_N^m$ is the same.
\end{rem}



\subsection{Abelian Surfaces: $n=2$}\label{subsectionabsurf2}
Now we will apply the computations of Section \ref{subsectionabsurfgeom2} and obtain similar stability results as in the case of regular surfaces in \cite{Wan12} (cf.\ Section \ref{subsectiontaut}). We proceed similarly as in the proofs in Sections 3 and 4 in \cite{Wan12} and will first exclude destabilising line subbundles. Note that for a p.p.a.s.\ satisfying ($\star$) we can write every line bundle $\mc{L}$ on $\widetilde{A\times A}$ as
\[\mc{L}\simeq r_1^\star\mc{M}_1\otimes r_2^\star\mc{M}_2\otimes\sigma^\star\mc{O}(bs^\star H)\otimes\mc{O}(cD)\]
with $\mc{M}_i\in\Pic A$ and $b,c\in\ZZ$.

\begin{prop}\label{propabstab}
Let $A$ be a p.p.a.s.\ satisfying $(\star)$ and let $\mc{F}$ be a $\mu_H$-stable vector bundle on $A$ of rank $r$ and first Chern class $c_1(\mc{F})=fH,$ $f\in\ZZ.$ Then $r_1^\star\mc{F}$ does not contain any line bundle $\mc{L}\simeq r_1^\star\mc{M}_1\otimes r_2^\star\mc{M}_2\otimes\sigma^\star\mc{O}(bs^\star H)\otimes\mc{O}(cD)$ with $c_1(\mc{M}_i)=a_iH,$ $a_i,b,c\in\ZZ$ satisfying the \em tautological destabilising condition \em
\[a_1+a_2+b\geq \frac{f}{r},\]
but in the case $\mc{F}\simeq\mc{M}_1,$ $\mc{M}_2\simeq \mc{O}_A,$ $b=c=0.$
\end{prop}
\begin{rem}
Note that the \em tautological destabilising condition \em does not contain $c$ at all. This is in perfect analogy to \cite[Cor.\ 3.2]{Wan12}.
\end{rem}

\begin{proof} We consider a line bundle $\mc{L}\simeq r_1^\star\mc{M}_1\otimes r_2^\star\mc{M}_2\otimes\sigma^\star\mc{O}(bs^\star H)\otimes\mc{O}(cD)$ with $c_1(\mc{M}_i)=a_iH$ and $a_i,b,c\in\ZZ$. Note that since $\sigma^\star\mc{O}(bs^\star H)$ comes from $A\times A,$ we can assume --- analogously to \cite[Lem.\ 4.1]{Wan12} --- that $c=0$. Thus it is enough to show that
\begin{equation}
\Hom_{A\times A}(\pi_1^\star\mc{M}_1\otimes\pi_2^\star\mc{M}_2\otimes s^\star\mc{O}(bH),\pi_1^\star\mc{F})\label{eqabhom}
\end{equation}
vanishes. To prove this we use adjunction ($\pi_1^\star\dashv\pi_{1\star},$ $\pi_2^\star\dashv\pi_{2\star}$ and $s^\star\dashv s_\star$) to get the following three different representations of this vector space:
\begin{eqnarray}
(\ref{eqabhom})&\cong&\Hom_A\big{(}\mc{M}_1,\mc{F}\otimes\pi_{1\star}(\pi_2^\star\mc{M}_2^\vee\otimes s^\star\mc{O}(-bH))\big{)} \label{eqabpi1}\\[4pt]&\cong&
\Hom_A\big{(}\mc{M}_2,\pi_{2\star}(\pi_1^\star(\mc{F}\otimes\mc{M}_1^\vee)\otimes s^\star\mc{O}(-bH))\big{)} \notag\\[4pt]&\cong &
\Hom_A\big{(}\mc{O}(bH),s_\star(\pi_1^\star(\mc{F}\otimes\mc{M}_1^\vee) \otimes\pi_2^\star\mc{M}_2^\vee)\big{)}.\label{eqabmu}
\end{eqnarray}
According to these three representations we shall consider three cases.

a) $a_2+b\geq0$: The restriction of $\pi_2^\star\mc{M}_2^\vee\otimes s^\star\mc{O}(-bH)$ to a fibre $\pi_1^{-1}(x)$ is isomorphic to $\mc{M}_2^\vee\otimes t^\star_x\mc{O}(-bH)$, where $t_x\colon A\rightarrow A$ is the translation by $x$. This is a line bundle with first Chern class $-(a_2+b)H$ on $A$. Thus if $a_2+b>0$, the space of global sections on the fibres is trivial and so is (\ref{eqabpi1}). If $a_2+b=0$ and $b\neq0$, then $\cH^0(\mc{M}_2^\vee\otimes t^\star_x\mc{O}(-bH))$ is zero outside a finite number of $x\in A$. (Note that since $H$ is an ample class, we have $\#\{x\in A\mid t_x^\star\mc{O}(-bH)\simeq\mc{M}_2^\vee\}<\infty$.) Hence $\pi_{1\star}(\pi_2^\star\mc{M}_2^\vee\otimes s^\star\mc{O}(-bH))$ would have finite support but since it is torsion-free, it vanishes. The remaining case is $a_2=b=0$. Now $\cH^0(\mc{M}_2^\vee\otimes t^\star_x\mc{O}(-bH))=\cH^0(\mc{M}_2^\vee)$ vanishes but in the case $\mc{M}_2\simeq\mc{O}_A$. Furthermore, (\ref{eqabpi1}) equals $\Hom(\mc{M}_1,\mc{F})$. The destabilising condition yields $a_1\geq \frac{f}{r}$ which implies $\Hom(\mc{M}_1,\mc{F})=0$ but in the case $\mc{F}\simeq\mc{M}_1$.

b) $a_2<0$:  Similar to above we consider the restriction of $\pi_1^\star(\mc{F}\otimes\mc{M}_1^\vee)\otimes s^\star\mc{O}(-bH)$ to a fibre $\pi_2^{-1}(x)$. Taking global sections, this yields
\[\cH^0(\mc{F}\otimes\mc{M}_1^\vee\otimes t_x^\star\mc{O}(-bH))\cong \Hom(\mc{M}_1\otimes t_x^\star\mc{O}(bH),\mc{F}),\]
which, by the stability of $\mc{F}$, has to vanish since the destabilising condition implies $a_1+b>\frac{f}{r}$.

c) $b<0$: Analogously to b) we now use (\ref{eqabmu}). This time the destabilising condition yields $a_1+a_2>\frac{f}{r}$.\end{proof}

From Lemma \ref{lemabsurf2slope} and Proposition \ref{propabstab} above we can deduce:

\begin{cor}
Assume $\mc{F}\not\simeq\mc{O}_A.$ Then for $N\gg 0$ there are no $\mu_{H_N^m}$-destabilising line subbundles in $\tamc{F}.$
\end{cor}

\begin{proof} Let $\mc{L}'$ be a destabilising line subbundle of $\tamc{F}$. We write its pullback $\mc{L}:=\psi^\star\mc{L}'$ as $\mc{L}=\sigma^\star(\mc{M}^{\boxtimes 2}\otimes\mc{O}(b s^\star H))\otimes\mc{O}(cD)$ with $c_1(\mc{M})=aH$, $a\in\ZZ.$ As usual, by adjunction we get a homomorphism $\mc{L}\rightarrow r_1^\star\mc{F}$. By Lemma \ref{lemabsurf2slope} the destabilising condition on $\mc{L}$ yields
\[2a+b\geq \frac{f}{r},\]
which implies that $\mc{L}$ satisfies the \em tautological destabilising condition \em from Proposition \ref{propabstab}.\end{proof}

This result on destabilising line subbundles suffices to prove the stability of rank two locally free tautological sheaves. We use the same argument as in the proof of Proposition \ref{propregsurflb} to generalise to the torsion-free case.

\begin{thm}\label{thmabsurfcasesr1}
Let $\mc{F}$ be a rank one torsion-free sheaf on $A$ satisfying $det\mc{F}\not\simeq\mc{O}_A.$ Then the rank two tautological sheaf $\tamc{F}$ is $\mu_{H_N^m}$-stable for sufficiently large $N.$
\end{thm}

Next, we prove the analogue of Theorem \ref{thmregsurfcasesr2}. Since the proof is almost the same, we only touch upon the crucial parts.

\begin{thm}\label{thmabsurfcasesr2}
Let $\mc{F}$ be a rank two $\mu_H$-stable sheaf on $A$ and assume $\det\mc{F}\not\simeq\mc{O}_A.$ Then for $N$ sufficiently large $\tamc{F}$ is a $\mu_{H^m_N}$-stable rank four sheaf on $\ta{A}.$
\end{thm}

\begin{proof} As above we may assume that $\mc{F}$ is locally free. We write $c_1(\mc{F})=fH$, $f\in \ZZ.$ Let $\mc{E}$ be the maximal destabilising sheaf of $\tamc{F}.$ Write $c_1(\mc{E})=eH_{\ta{A}}+gm^\star H+a\delta$, $e,g,a\in\ZZ$. We only consider the case that $\mc{E}$ is of rank two. We use the same notation as in the proof of Theorem \ref{thmregsurfcasesr2}:

If $\rk\ker\beta=0$, we must have that 
\[r_1^\star\det\mc{F}\otimes\psi^\star\det\mc{E}^\vee\]
has a section. From this we deduce that either $a<0$ and $\det\mc{F}\simeq\mc{O}_A$ (which we excluded) or $a\leq0$ and the class
\[(f-e)H\otimes 1 - 1\otimes e - g\sigma^\star s^\star H\]
on $\widetilde{A\times A}$ is effective and nonzero. This time the evaluation against the polarisation $\psi^\star H^m_N$ gives a contradiction to the destabilising condition on $\mc{E}$.

If $\rk\ker\beta=1$, we write $c_1(\im\beta)=l_1H\otimes1+1\otimes l_2+h\sigma^\star s^\star H+bD$ with $l_1,l_2,h,b\in\ZZ$. The semistability of $\mc{E}$ yields
\[2e+g\leq 2(l_1+l_2+h)\]
and the destabilising condition on $\mc{E}$ implies
\[2e+g\geq f.\]
Thus we found a line bundle in $r_1^\star\mc{F}$ with
\[2(l_1+l_2+h)\geq f\]
contradicting Proposition \ref{propabstab}.\end{proof}

\subsection{Abelian Surfaces: $n=3$}\label{subsectionabsurf3}
In this section we prove the following result:

\begin{thm}\label{thmabsurf3}
Let $A$ be a p.p.a.s.\ satisfying $(\star)$ and let $\mc{F}$ be a rank one torsion-free sheaf on $A$ satisfying $\det\mc{F}\not\simeq\mc{O}_A.$ Then for all $N$ sufficiently large $\tbmc{F}$ is $\mu_{H_N^m}$-stable.
\end{thm}

We denote the projection $A^3\rightarrow A$ to the $i$-th factor by $\pi_i$ and --- in analogy to the case $n=2$ --- we begin by analysing line subbundles of the sheaf $\pi_1^\star\mc{F}$ on $A^3$.

\begin{prop}\label{propabsurf3lb}
Let $\mc{F}$ be a $\mu_H$-stable vector bundle on $A$ of rank $r$ and first Chern class $fH,$ $f\in\ZZ.$ Then $\pi_1^\star\mc{F}$ does not contain any line subbundles of the form
\[\mc{L}=\mc{M}^{\boxtimes 3}\otimes s_3^\star\mc{O}(bH),\] $\mc{M}\in\Pic A,$ $c_1(\mc{M})=aH,$ $a,b\in\ZZ,$ satisfying the \em tautological destabilising condition \em
\[3a+b\geq \frac{f}{r},\]
but in the case $a=b=0,$ $\mc{M}\simeq\mc{F}\simeq\mc{O}_A.$
\end{prop}

\begin{proof} Again, we distinguish three cases:

a) $2a+b\geq 0$: We push forward along $\pi_1$:
\begin{eqnarray}
\Hom(\mc{L},\pi_1^\star\mc{F})&\cong& \cH^0(\pi_1^\star\mc{F}\otimes\mc{L}^\vee)\notag\\
&\cong&
\cH^0(\mc{F}\otimes\mc{M}^\vee\otimes \pi_{1\star}(\underbrace{    \pi_2^\star\mc{M}^\vee\otimes \pi_3^\star\mc{M}^\vee\otimes s_3^\star\mc{O}(-bH)}_{=:\mc{G}})).\label{eqabsurf3erste}
\end{eqnarray}
Restricting $\mc{G}$ to a fibre $\pi_1^{-1}(x),$ $x\in A$ yields (we identify $\pi_1^{-1}(x)=\{x\}\times A^2$, keep the notation for the projections $\pi_i,$ $i=2,3$ but denote with $s$ the multiplication of the second two factors):
\[\pi_2^\star\mc{M}^\vee\otimes \pi_3^\star\mc{M}^\vee\otimes s^\star t_x^\star\mc{O}(-bH).\]
The class of this line bundle on $A^2$ has degree $-(2a+b)$ with respect to the polarisation $H^{\boxplus 2}+s^\star H$ (cf.\ Lemma \ref{lemabsurf2slope}). Thus we get a contradiction but in the case $a=b=0$ and $\mc{M}\simeq \mc{O}_A$. In this case we have $(\ref{eqabsurf3erste})\cong\Hom(\mc{O}_A,\mc{F}),$ which vanishes by the stability of $\mc{F}$ but in the case $\mc{F}\simeq \mc{O}_A$.

b) $a<0$: we push forward along $\pi_2$:
\begin{eqnarray*}
\Hom(\mc{L},\pi_1^\star\mc{F})&\cong& \cH^0(\pi_1^\star\mc{F}\otimes\mc{L}^\vee)\\
&\cong&
\cH^0(\mc{M}^\vee\otimes \pi_{2\star}(\underbrace{\pi_1^\star(\mc{F}\otimes\mc{M}^\vee)\otimes \pi_3^\star\mc{M}^\vee\otimes s^\star\mc{O}(-bH)}_{=:\mc{H}})).
\end{eqnarray*}
Again, we restrict $\mc{H}$ to a fibre $\pi_2^{-1}(x),$ $x\in A$ and then take global sections:
\begin{eqnarray*}
\cH^0(\mc{H}|_{\pi_2^{-1}(x)})&\cong&\cH^0(\pi_1^\star(\mc{F}\otimes\mc{M}^\vee)\otimes \pi_3^\star\mc{M}^\vee\otimes  s^\star t_x^\star\mc{O}(-bH))\\
&\cong&\Hom(\underbrace{\pi_1^\star\mc{M}\otimes \pi_3^\star\mc{M}\otimes s^\star t_x^\star\mc{O}(bH)}_{=:\mc{L}'},\pi_1^\star\mc{F}).
\end{eqnarray*}
But this leads to a contradiction to Proposition \ref{propabstab}: The first Chern class of $\mc{L}'$ is $aH^{\boxplus2}+bs^\star H$ and by assumption we have
\[\frac{f}{r}\leq 3a+b<2a+b.\]

c) $b<0$: We push forward along $s_3$:
\begin{eqnarray*}
\Hom(\mc{L},\pi_1^\star\mc{F})&\cong& \cH^0(\pi_1^\star\mc{F}\otimes\mc{L}^\vee)\\
&\cong&
\cH^0(\mc{O}(-bH)\otimes s_{3\star} (\underbrace{\pi_1^\star(\mc{F}\otimes\mc{M}^\vee) \otimes    \pi_2^\star\mc{M}^\vee\otimes \pi_3^\star\mc{M}^\vee}_{=:\mc{E}})).
\end{eqnarray*}
The fibre of $s_3$ over a point $x\in A$ can be identified with $A^2$ as follows:
\[A^2\rightarrow s_3^{-1}(x),\quad(y,z)\mapsto (y,z,x-(y+z)).\]
Under this identification $\pi_1$ and $\pi_2$ remain the same and $\pi_3$ is replaced by $t_x\circ\iota\circ s,$ where $\iota\colon A\rightarrow A$ is the inverse. Thus we see that the restriction of $\mc{E}$ to $s_3^{-1}(x)$ is isomorphic to
\[\pi_1^\star\mc{F}\otimes\underbrace{\pi_1^\star\mc{M}^\vee \otimes\pi_2^\star\mc{M}^\vee\otimes s^\star\iota^\star t_x^\star\mc{M}^\vee}_{=:\mc{L}''}.\]
Now, exactly as in b) above we get a contradiction to Proposition \ref{propabstab}: The first Chern class of $\mc{L}''$ is $a(H^{\boxplus}+s^\star H)$ and by assumption we have
\[\frac{f}{r}\leq 3a+b<3a.\qedhere\]
\end{proof}

\begin{proof}[Proof of Theorem \ref{thmabsurf3}] First of all we can reduce to the case that $\mc{F}$ is a line bundle in the same way as in the proof of Theorem \ref{thmregsurfcasesr1}. Next, let $\mc{L}\subset\tbmc{F}$ be a destabilising line subbundle. As usual we see that this yields a nontrivial element in 

\[\Hom(\mc{L},\tbmc{F})\cong \Hom(\mc{L},\psi_{3\star}\sigma_3^\star q^\star\mc{F}) \cong \Hom(\psi_3^\star\mc{L},\sigma_3^\star q^\star\mc{F}).\]

By the same reasoning as in Section \ref{sectionhighern} we may assume that $\mc{L}\simeq \mc{M}_{\tb{A}}\otimes m_3^\star\mc{O}(bH)),$ $c_1(\mc{M})=aH,$ $a,b\in\ZZ$ (we do not have any term of the form $\mc{O}(c\delta_3)$). Thus $\psi_3^\star\mc{L}$ descends to a line bundle $\mc{L}'$ on $A\times\ta{A}$. Furthermore, the pullback $(\id_A\times\psi)^\star\mc{L}'$ on $A\times\widetilde{A\times A}$ descends to a line bundle $\mc{L}''$ on $A^3$. Thus we have
\begin{eqnarray*}
\Hom(\psi_3^\star\mc{L},\sigma_3^\star q^\star\mc{F})&=&\Hom(\mc{L}',q^\star\mc{F})\\[4pt]
&\subseteq&\Hom((\id_A\times\psi)^\star\mc{L}',(\id_A\times\sigma)^\star\pi_1^\star\mc{F})\\[4pt]
&=&\Hom(\mc{L}'',\pi_1^\star\mc{F}).
\end{eqnarray*}
Looking more closely we see that $\mc{L}''\simeq \mc{M}^{\boxtimes3}\otimes s_3^\star\mc{O}(bH).$
Altogether we end up with a homomorphism on $A^3$:
\[\mc{M}^{\boxtimes 3}\otimes s_3^\star\mc{O}(bH)\rightarrow \pi_1^\star\mc{F}.\]
Now by Lemma \ref{lemabsurfgeom5} the destabilising condition on $\mc{L}$ reads
\[3a+b\geq f,\]
which exactly corresponds to the \em tautological destabilising condition \em of Proposition \ref{propabsurf3lb}.\end{proof}

\subsection{Restriction to the Kummer Surface}\label{subsectionkum}
In Section \ref{subsectionabsurf2} we proved the stability of tautological sheaves on the Hilbert scheme of two points on an abelian surface $A$. This Hilbert scheme contains the Kummer surface $\Km A$ associated with $A$. Recall that $\Km A$ is a $K3$ surface. In this section we shall prove the stability of the restriction of certain tautological sheaves to the Kummer surface. Throughout this Section we fix an arbitrary(!) polarised abelian surface $(A,H).$\\

Let us recall the famous Kummer construction. The natural involution on $A$ given by the inverse $\iota\colon a\mapsto -a$ has $16$ fixed points (which are exactly the $16$ two-torsion points). Thus if we blow up the $16$ $A_1$-singularities in $A/\iota$, we obtain a smooth surface which can easily be shown to be $K3$. It is the so-called \em Kummer surface $\Km A$ associated with $A$\em.

An alternative construction of $\Km A$ goes as follows: Let $b\colon\tilde{A}\rightarrow A$ denote the simultaneous blowup of all fixed points of the involution $\iota$ on $A$ and denote by $E_1,\dots,E_{16}$ the exceptional divisors. On $\tilde{A}$ we still have an involution which fixes the $E_l$ pointwise. We consider the quotient $\tau\colon\tilde{A}\rightarrow \Km A$ which is a degree two covering onto the associated Kummer surface. By \cite[VIII Prop.\ 5.1]{BHPV04} we have a monomorphism
\[\alpha=\tau_! b^\star\colon \cH^2(A,\ZZ)\rightarrow \cH^2(\Km A,\ZZ)\]
satisfying
\[\alpha(x)\alpha(y)=2xy \text{ for all }x,y\in \cH^2(A,\ZZ).\]
We have an inclusion of finite index
\[\alpha(\NS A)\oplus \bigoplus_{l=1}^{16}\ZZ N_l\subset\NS(\Km A),\]
where $N_l=\tau(E_l)$. It is well known that $E_l^2=-1$ and $N_l^2=-2$. Furthermore, the class $\sum_lN_l$ is $2$-divisible and we have $\tau^\star(\frac{1}{2}\sum_lN_l)=\sum_lE_l$ and $\tau^\star N_l=2E_l$.

Finally, we have
\begin{equation}\label{eqihsexakum}
\NS \tilde{A}\cong b^\star\NS A\oplus \bigoplus_{l=1}^{16}\ZZ E_l\quad\text{and}\quad \Pic^0\tilde{A}\cong b^\star\Pic^0A.\end{equation}

We consider the following diagram:
\[\xymatrix{
A \ar@{^{(}->}[r]^(.4)u & A\times A \ar[rd]^s  \\
\tilde{A} \ar[u]^b \ar@{^{(}->}[r]^(.4){\widetilde{j}} \ar[d]_(.4)\tau & \widetilde{A\times A} \ar[u]^\sigma \ar[r] \ar[d]_(.4)\psi  & A\\
\Km A \ar@{^{(}->}[r]^j & \ta{A} \ar[ru]_\rho &,
}\]
where we define $u(x):=(x,-x)$ and $\widetilde{j}$ as the pullback of $u$ and $\sigma$. With this convention we have $\pi_1\circ u=\id_A$. Finally, note that since $\psi$ is flat, we have $j^\star\circ\psi_\star\simeq \tau_\star\circ\widetilde{j}^\star$.

We define a class
\[H_N^\Km:=N\alpha(H)-\frac{1}{2}\sum_lN_l\]
on $\Km A,$ which is ample for sufficiently large $N.$ Recall that on $\ta{A}$ we defined two different polarisations $H_N:=NH_{\ta{A}}-\delta$ and $H_N^m:=N(H_{\ta{A}}+m^\star H)-\delta$.

\begin{lem}
We have
\begin{eqnarray*}
H_N^\Km = j^\star H_N=j^\star H_N^m.
\end{eqnarray*}
\end{lem}

\begin{proof} It is clear that $j^\star m^\star H=0$ since $\Km A$ is a fibre of $m$. Also observe that the class of the intersection of the exceptional divisor of the Hilbert$-$Chow morphism $\rho$ with $\Km A$ is exactly $\sum_lN_l$.\end{proof}

\begin{defi}
Let $\mc{F}$ be a sheaf on $A$. We set
\begin{eqnarray*}
\kmmc{F}:=\tau_\star b^\star\mc{F}.
\end{eqnarray*}
\end{defi}

\begin{lem}
The sheaf $\kmmc{F}$ is the restriction of the tautological sheaf $\tamc{F}$:
\begin{eqnarray*}
j^\star\tamc{F}\simeq \kmmc{F}.
\end{eqnarray*}
\end{lem}

\begin{proof} We have
\begin{eqnarray*}
\begin{array}{lclclclclclr}
j^\star\tamc{F}&=&j^\star\psi_\star \sigma^\star\pi_1^\star\mc{F} &\simeq&\tau_\star\widetilde{j}^\star \sigma^\star\pi_1^\star\mc{F} &\simeq&\tau_\star b^\star u^\star\pi_1^\star\mc{F}&\simeq &\tau_\star b^\star\mc{F}&=&\kmmc{F}.&\qedhere
\end{array}
\end{eqnarray*}
\end{proof}

Now we want to prove the stability of $\kmmc{F}$ in the case that $\mc{F}$ is of rank one or two. The method is completely analogous to the one used in the preceding sections. Thus we will leave a few details to the reader. As in the previous cases we begin with the analysis of line subbundles in the pullback $b^\star\mc{F}$:

\begin{prop}\label{propabsurfkum}
Let $\mc{F}$ be a $\mu_H$-stable sheaf on $A$ of rank $r$ and first Chern class $f\in\NS A.$ Then $b^\star\mc{F}$ does not contain any line bundle $\mc{L}'=b^\star\mc{G}\otimes\mc{O}(\sum_la_lE_l)$ with $\mc{G}\in\Pic(A),$ $c_1(\mc{G})=g'$ satisfying
\[H.g'\geq\frac{1}{r}H.f\]
but in the case $r=1,$ $\mc{G}\simeq \mc{F}$.
\end{prop}

\begin{proof} As usual we may assume that $\mc{F}$ is locally free. We want to show that
\[\Hom_{\tilde{A}}(b^\star\mc{G}\otimes\mc{O}(\sum_la_lE_l),b^\star\mc{F})\]
vanishes. Using adjunction ($b^\star\dashv b_\star$) and a similar induction argument as in the proof of \cite[Prop.\ 3.1]{Wan12}, we see that it is enough to prove that
\[\Hom_A(\mc{G},\mc{F})=0.\]
This easily follows from the stability of $\mc{F}$ if $\mc{F}\not\simeq\mc{G}$.\end{proof}

Next, we will show that Proposition \ref{propabsurfkum} implies that there are no destabilising line subbundles in $\kmmc{F}$. We only need to calculate slopes.

\begin{lem}
Let $\mc{F}$ be a sheaf on $A$ of rank $r$ and first Chern class $f.$ We have
\begin{eqnarray*}
c_1(\tau_\star b^\star\mc{F})&=&\alpha(f)-\frac{r}{2}\sum_lN_l.
\end{eqnarray*}
\end{lem}

\begin{proof} We have $c_1(\omega_{\tilde{A}})=\sum_lE_l$. Thus the Grothendieck$-$Riemann$-$Roch theorem reads
\begin{eqnarray*}
\begin{array}{lclclr}
\ch(\tau_\star b^\star\mc{F})&=& \tau_\star(\ch(b^\star\mc{F})\td_\tau)&=&
\tau_\star((r,b^\star f,\dots)(1,-\frac{1}{2}\sum_lE_l,\dots))\\[4pt]
&&&=&\tau_\star(r,b^\star f-\frac{r}{2}\sum_lE_l,\dots).&\qedhere
\end{array}
\end{eqnarray*}
\end{proof}

Let $\mc{L}$ be a line bundle on $\Km A$. By equation (\ref{eqihsexakum}) in Section \ref{subsectionihsexamples} there is a line bundle $\mc{G}$ on $A$ and integers $a_l$ such that
\[\tau^\star\mc{L}\simeq b^\star\mc{G}\otimes\mc{O}(\sum_la_lE_l).\]
Set $g:=c_1(\mc{G})$. Note that since $\mc{L}$ comes from $\Km A$, the line bundle $\mc{G}$ has to be symmetric, i.e.\ $\iota^\star\mc{G}\simeq\mc{G}.$

\begin{cor}
Let $\mc{L}$ be a line bundle on $\Km A$ as above. We have
\begin{eqnarray*}
\mu_{H_N^\Km}(\kmmc{F})&=&\frac{1}{r}NH.f-4 \text{ and}\\[4pt]
\mu_{H_N^\Km}(\mc{L})&=&NH.g+\frac{1}{2}\sum_la_l.
\end{eqnarray*}
\end{cor}

\begin{proof}
We pullback all classes to $\tilde{A}$: Note that $\tau^\star(\frac{1}{2}\sum_lN_l)=\sum_lE_l$ and $\tau^\star\alpha(f)=2b^\star f$ for all $f\in\NS A$. Thus we have $\alpha(f).\alpha(H)=2f.H.$ Furthermore, we have $\big{(}\sum_lE_l\big{)}^2=16\cdot(-1)=-16$ and $\big{(}\sum_lE_l\big{)}\big{(}\sum_la_lE_l\big{)}=-\sum_la_l.$ Finally, we have to divide everything by two because we pulled back along a degree two covering.
\end{proof}

\begin{cor}\label{corabsurfkum}
Let $\mc{F}$ be a non-symmetric (i.e.\ $\iota^\star\mc{F}\not\simeq\mc{F}$) $\mu_H$-stable sheaf on $A.$ Then $\kmmc{F}$ does not contain $\mu_{H_N^\Km}$-destabilising line subbundles for all $N\gg0.$
\end{cor}

\begin{proof} Let $\mc{L}$ be a destabilising line subbundle of $\kmmc{F}$. Again, we can write $\tau^\star\mc{L}\simeq b^\star\mc{G}\otimes\mc{O}(\sum_la_lE_l)$ for a symmetric line bundle $\mc{G}\in\Pic A$. The destabilising condition yields
\[H.g\geq \frac{1}{r}H.f.\]
As usal we use adjunction $\tau^\star\dashv \tau_\star$ to obtain a homomorphism $\tau^\star\mc{L}\rightarrow b^\star\mc{F}$. This gives a contradiction to Proposition \ref{propabsurfkum} but in the case $r=1,$ $\mc{G}\simeq \mc{F}$. But this cannot be since $\mc{F}$ was chosen not to be symmetric.\end{proof}

We immediately deduce:

\begin{thm}\label{thmabsurfkumr1}
Let $\mc{F}$ be a non-symmetric rank one torsion-free sheaf on $A.$ Then for all $N$ sufficiently large, $\kmmc{F}=\tau_\star b^\star\mc{F}$ is a rank two $\mu_{H_N^\Km}$-stable sheaf.
\end{thm}

\begin{exa}
We apply the theorem to the case $c_1(\mc{F})=0$. Denote by $\hat{A}$ the dual abelian variety and by $\hat{A}[2]$ its two-torsion points. The assignment $\Pic^0 A\ni\mc{F}\mapsto\kmmc{F}$ gives a map
\[\hat{A}\setminus\hat{A}[2]\rightarrow \mc{M},\]
where $\mc{M}:=\mc{M}_{H_N^\Km}(v)$ is the moduli space of $H_N^\Km$-stable sheaves with
\[v=(2,-\frac{1}{2}\sum_lN_l,-2).\]
Note that $v^2=0$ and the first Chern class $-\frac{1}{2}\sum_lN_l$ is primitive. Hence $\mc{M}$ is smooth of dimension two. Since $\kmmc{F}\simeq\km{(\iota^\star\mc{F})},$ this map is two-to-one. Furthermore, let us consider the case that $\mc{F}$ is symmetric, i.e.\ $\mc{F}\in\hat{A}[2].$ We concentrate on the case $\mc{F}=\mc{O}_A.$ We have extensions
\[0\rightarrow \mc{O}(-\frac{1}{2}\sum_lN_l) \rightarrow \mc{E} \rightarrow \mc{O} \rightarrow 0.\]
The sheaf $\kmmc{O}_X$ is isomorphic to the trivial extension $\mc{O}(-\frac{1}{2}\sum_lN_l)\oplus\mc{O}$ (cf.\ \cite[Lem.\ 17.2]{BHPV04}), which is not stable. On the other hand one can show that every nontrivial extension is $\mu_H$-stable, which can be proven as in Lemma \ref{leminvstab}. The vector space of extensions $\mc{E}$ is two-dimensional and thus we have a $\PP^1\subset\mc{M}$ parametrising the $\mc{E}$.
Altogether we see that $\mc{M}$ is isomorphic to the Kummer surface $\Km\hat{A}$ of the dual abelian surface $\hat{A}.$

If $\mc{F}$ has nontrivial first Chern class $f\in\NS A,$ we may choose a symmetric line bundle $\mc{L}$ satisfying $c_1(\mc{L})=-f$. Then $\mc{F}\otimes\mc{L}$ is in $\Pic^0(A)$ and
\[\km{(\mc{F}\otimes\mc{L})}\simeq \kmmc{F}\otimes \mc{O}(\alpha(-f)).\]
Thus the moduli space containing $\kmmc{F}$ is isomorphic to $\Km\hat{A},$ too.

But by \cite[Thm.\ 1.5]{GH98} the Kummer surfaces $\Km A$ and $\mc{M}\cong \Km\hat{A}$ are isomorphic.
\end{exa}

We finish the section by proving the analogue of Theorem \ref{thmregsurfcasesr2}.

\begin{thm}\label{thmabsurfkumr2}
Let $\mc{F}$ be a $\mu_H$-stable rank two sheaf on $A$ such that $\det\mc{F}$ is not symmetric. Then $\kmmc{F}$ is a $\mu_{H_N}$-stable rank four sheaf on $\Km A.$
\end{thm}

\begin{proof} We exactly imitate the proof of Theorem \ref{thmregsurfcasesr2}. Assume that $\mc{F}$ is locally free and let $f:=c_1(\mc{F})$. Let $\mc{E}$ be a reflexive semistable rank two subsheaf of $\kmmc{F}$ and write $c_1(\mc{E})=\alpha(e)+\sum_la_lN_l.$ The destabilising condition thus implies
\[2H.e\geq H.f.\]
We have a homomorphism $\beta\colon\tau^\star\mc{E}\rightarrow b^\star\mc{F}$. Again, the only difficult case is when $\ker\beta=0$: If the first Chern class of the $\mc{Q}:=\coker\beta$ is trivial, we see that the homological dimension of $\mc{Q}$ is $2$. Since $b^\star\mc{F}$ is locally free, this would contradict the fact that $\tau^\star\mc{E}$ is reflexive. Thus $\mc{Q}=0$ and $\beta$ has to be an isomorphism. But since $\tau^\star\mc{E}$ is symmetric and $\mc{F}$ is not, we are done.

If there is an effective divisor with first Chern class $c_1(\mc{Q})$, the line bundle
\[b^\star\det\mc{F}\otimes\tau^\star\det\mc{E}^\vee(-\sum_la_lN_l)\]
must have a section. Hence either $a_l<0$ $\forall l$ and $\det\mc{F}\simeq\mc{O}_A$ (which we excluded) or $a_l\leq0$ $\forall l$ and
\[H.f>2H.e\]
which contradicts the stability condition.\end{proof}

\subsection{Generalised Kummer Varieties of Dimension Four}\label{subsectiongenkum4}
Let $(A,H)$ be a p.p.a.s. satisfying ($\star$) and let $\tb{A}$ be the Hilbert scheme of three points on $A$. In this section we prove some results concerning the stability of the restriction of tautological sheaves to the four dimensional generalised Kummer variety $j\colon K_3(A)\hookrightarrow\tb{A}$.\\

We have an isomorphism
\[\NS(\tb{A})\cong\ZZ H_{\tb{A}}\oplus \ZZ m_3^\star H\oplus \ZZ\delta_3.\]
Restricting to $K_3(A),$ we obtain
\[\NS(K_3(A)\cong \ZZ j^\star H_{\tb{A}}\oplus \ZZ j^\star\delta_3.\]
Note that again $j^\star m_3^\star H=0$. Considering the polarisations $H_N=NH_{\tb{A}}-\delta_3$ and $H_N^m=N(H_{\tb{A}}+m_3^\star H)-\delta_3$ on $\tb{A},$ we define a polarisation \[H^K_N:=j^\star H_N=j^\star H_N^m=Nj^\star H_{\tb{A}}-j^\star\delta_3.\]

\begin{lem}\label{lemabsurfgen4slope}
We have
\begin{eqnarray*}
(H_N^K)^3\cdot j^\star \delta_3=0 + O(N^4).
\end{eqnarray*}
\end{lem}

\begin{proof} By definition of $H_N^K$ we have $(H_N^K)^3\cdot j^\star \delta_3=j^\star((H_N)^3\delta_3)$. Now the lemma follows from equation (\ref{eqpolaint2}) in Lemma \ref{lempolaintersect}.\end{proof}

\begin{prop}\label{propabsurfgen4}
Let $\mc{F}$ be a $\mu_H$-stable sheaf on $A$ of rank $r$ and first Chern class $fH.$ If $\mc{F}^{\vee\vee}\not\simeq\mc{O}_A,$ then for $N$ sufficiently large $\kmc{F}:=j^\star\tbmc{F}$ does not contain any $\mu_{H^K_N}$-destabilising subsheaves of rank one.
\end{prop}

\begin{proof} We may assume that $\mc{F}$ is locally free. Since all line bundles on $K_3(A)$ come from $\tb{A},$ we may assume that a destabilising line bundle is of the form $j^\star\mc{M}'$ with $\mc{M}'\in\Pic(\tb{A})$. Furthermore, since $j^\star\mu_3^\star H=0$, we may assume that we have no contribution of this summand in $\mc{M}'$ and by a similar reasoning as in \cite[Prop.\ 3.1]{Wan12} we can reduce to the case where we have no contribution of the $\delta_3$-summand neither. Thus there is, in fact, a line bundle $\mc{M}$ on $A$ such that $\mc{M}'\simeq\mc{M}_{\tb{A}}$. Let $l\in\ZZ$ be such that $lH=c_1(\mc{M})$.

We consider the following two diagrams:
\[\xymatrix{
A^{[2,3]} \ar[r]^(.45){\sigma_3} \ar[d]_(.45){\psi_3} & A\times\ta{A}\ar[d] & A\times\widetilde{A\times A} \ar[l]_\psi \ar[d]^(.45)\sigma &&
(\widetilde{K_3(A)},\widetilde{j}) \ar[d]_(.45){\psi_3'} \ar[r]^(.55){\sigma_3'}&(B,k)\ar[d]&(\widetilde{B},\widetilde{k}) \ar[d]^(.45){\sigma'} \ar[l]_{\psi'}\\
\tb{A} \ar[r]^{\rho_3} \ar[dr]_{m_3} & S^3A \ar[d]^(.45){\tilde{s}_3}& A^3 \ar[l] \ar[dl]^{s_3}&&
(K_3(A),j) \ar[r] \ar[dr] & \tilde{s}_3^{-1}(0) \ar[d]& (s_3^{-1}(0),g)\ar[l] \ar[dl]\\
& A&&&&\{0\}&.
}\]
The right diagram is a subdiagram of the left one. Where needed the symbols for the inclusion morphisms are added. For example, $(B,k)$ denotes $B:=\{(x,\xi)\mid x+m(\xi)=0\}$ together with the inclusion $k\colon B\hookrightarrow A\times \ta{A}$ and $\widetilde{K_3(A)}$ is the strict transform of $B$ along $\sigma_3$. In the left diagram we used the abreviations $\psi$ and $\sigma$ to denote $\id_A\times\psi$ and $\id_A\times\sigma,$ respectively.

Note that $\psi_3$ is flat outside codimension four and thus its restriction $\psi_3'$, too, because $\widetilde{K_3(A)}$ is the preimage of $K_3(A)$ under $\psi_3$. Thus for all sheaves $\mc{G}$ on $K_3(A)$ and $\mc{H}$ on $A^{[2,3]}$ we have an isomorphism
\begin{eqnarray*}
\Hom(\mc{G},j^\star\psi_{3\star}\mc{H})&\cong&\Hom(\mc{G},\psi_{3\star}'\widetilde{j}^\star\mc{H}).
\end{eqnarray*}
Furthermore, we have 
\[\psi_3'^\star j^\star\simeq\widetilde{j}^\star\psi_3^\star, \quad \widetilde{j}^\star\sigma_3^\star\simeq\sigma_3'^\star k^\star,\quad \psi'^\star k^\star\simeq \widetilde{k}^\star\psi^\star \quad\text{and}\quad \widetilde{k}^\star\sigma^\star\simeq \sigma'^\star g^\star.\]
Finally, we have
\begin{eqnarray*}
\psi_3^\star\mc{M}_{\tb{A}}&\simeq&\sigma_3^\star(\mc{M}\boxtimes \mc{M}_{\ta{A}})\quad\text{ and}\\[4pt]
\psi^\star(\mc{M}\boxtimes \mc{M}_{\ta{A}})&\simeq& \sigma^\star \mc{M}^{\boxtimes 3}. 
\end{eqnarray*}
We denote the projections $A^3\rightarrow A$ by $\pi_i$. We have
\begin{eqnarray*}
\begin{array}{clcl}
&\Hom(j^\star\mc{M},j^\star\tbmc{F})&\cong& \Hom(j^\star\mc{M},j^\star\psi_{3\star}\sigma_3^\star \pi_1^\star\mc{F})\\[4pt]
\cong& \Hom(j^\star\mc{M},\psi'_{3\star}\widetilde{j}^\star\sigma_3^\star \pi_1^\star\mc{F})& \cong&
\Hom(\psi_3'^\star j^\star\mc{M},\widetilde{j}^\star\sigma_3^\star \pi_1^\star\mc{F})\\[4pt]
\cong& \Hom(\widetilde{j}^\star\psi_3^\star\mc{M},\widetilde{j}^\star\sigma_3^\star \pi_1^\star\mc{F})&\cong&
\Hom(\widetilde{j}^\star\sigma_3^\star(\mc{M}\boxtimes \mc{M}_{\ta{A}}),\widetilde{j}^\star\sigma_3^\star \pi_1^\star\mc{F})\\[4pt]
\cong& \Hom(\sigma_3'^\star k^\star(\mc{M}\boxtimes \mc{M}_{\ta{A}}),\sigma_3'^\star k^\star \pi_1^\star\mc{F})&\cong&
\Hom(k^\star(\mc{M}\boxtimes \mc{M}_{\tb{A}}),k^\star \pi_1^\star\mc{F})\\[4pt]
\subseteq& \Hom(\psi'^\star k^\star(\mc{M}\boxtimes \mc{M}_{\tb{A}}),\psi'^\star k^\star \pi_1^\star\mc{F})&\cong&
\Hom(\widetilde{k}^\star\psi^\star(\mc{M}\boxtimes \mc{M}_{\tb{A}}),\widetilde{k}^\star\psi^\star \pi_1^\star\mc{F})\\[4pt]
\cong& \Hom(\widetilde{k}^\star\sigma^\star\mc{M}^{\boxtimes 3},\widetilde{k}^\star\sigma^\star \pi_1^\star\mc{F})&\cong&
\Hom(\sigma'^\star g^\star\mc{M}^{\boxtimes 3},\sigma'^\star g^\star \pi_1^\star\mc{F})\\[4pt]
\cong& \Hom(g^\star\mc{M}^{\boxtimes 3},g^\star \pi_1^\star\mc{F})&\cong&
\cH^0(g^\star(\mc{M}^{\vee\boxtimes 3}\otimes \pi_1^\star\mc{F})).
\end{array}
\end{eqnarray*}
In order to proceed, we choose an isomorphism $s_3^{-1}(0)\cong A^2$ by sending $(x,y,z)$ to $(x,z)$ and denote the projections $A^2\rightarrow A$ by $\hat{\pi}_i$. In this picture we have the identifications: $\pi_1\circ g = \hat{\pi}_1,$ $\pi_2\circ g \corres \iota\circ s$ and $\pi_3\circ g \corres \hat{\pi}_2$. (Recall that $s\colon A^2\rightarrow A$ denotes the summation map.) Thus pushing forward along $p_1$ ($p_2$ in the second line), we have
\begin{eqnarray}
\cH^0(g^\star(\mc{M}^{\vee\boxtimes 3}\otimes \pi_1^\star\mc{F}))&\cong& \cH^0(\mc{F}\otimes\mc{M}^\vee\otimes \hat{\pi}_{1\star}(\hat{\pi}_2^\star\mc{M}^\vee\otimes s^\star\iota^\star\mc{M}^\vee)) \label{eqgenkum41}\\[4pt]
&\cong& \cH^0(\mc{M}^\vee\otimes \hat{\pi}_{2\star}\label{eqgenkum42}(\hat{\pi}_1^\star(\mc{F}\otimes\mc{M}^\vee)\otimes s^\star\iota^\star\mc{M}^\vee)).
\end{eqnarray}
By Lemma \ref{lemabsurfgen4slope} the destabilising condition of $j^\star\mc{M}$ in $j^\star\tbmc{F}$ implies
\[l\geq \frac{f}{3r}.\]

If $l\geq 0$, we see that the right hand side of (\ref{eqgenkum41}) vanishes but in the case $\mc{M}\simeq \mc{O}_A\simeq\mc{F}.$

If $l<0$, the destabilising condition implies
\[2l>\frac{f}{r}.\]
In this case the right hand side of (\ref{eqgenkum42}) has to vanish by Proposition \ref{propabstab}.\end{proof}

As usual, from Proposition \ref{propabsurfgen4} we can deduce the stability of rank three restricted tautological sheaves associated with rank one sheaves:

\begin{thm}\label{thmabsurfgen4}
Let $\mc{F}$ be a torsion-free rank one sheaf on $A.$ Assume $\det\mc{F}\not\simeq\mc{O}_A.$ Then for all sufficiently large $N$ the sheaf $j^\star\tbmc{F}$ is $\mu_N^K$-stable.
\end{thm}

\section{Deformations and Moduli Spaces of Tautological Sheaves}\label{sectiondef}
This chapter collects a few results on different aspects of the behaviour of tautological sheaves under deformations.

\subsection{Deformations of Tautological Sheaves} \label{subsectiondeftaut}
In this section we will make the following general assumption:
\[\left(\begin{array}{c}
X\text{ is a }K3\text{ surface and }\mc{F}\text{ a stable sheaf on }X\text{ with Mukai vector }v\text{ such that}\\
\text{ for every sheaf }\mc{G}\in\mc{M}^s(v)\text{ the associated tautological sheaf }\tnmc{G}\text{ is also stable.}
\end{array}\right)\] 

Note that in the cases where the stability of tautological sheaves has been explicitly proven the tautological sheaf associated with a sheaf $\mc{F}$ is stable if and only if it is true for every other $\mc{G}$ in the same moduli space. (We are only considering sheaves on $K3$ surfaces.)

Denote by $\tn{v}\in \cH^\ast(\tn{X},\QQ)$ the Mukai vector of $\tnmc{F}$. The assignment
\[\mc{F}\mapsto\tnmc{F}\]
yields a morphism
\[\tn{[-]}\colon\mc{M}^s(v)\rightarrow\mc{M}^s(\tn{v}).\]
We shall mainly discuss the case $n=2$. Let us prove the following lemma which shows that $\ta{[-]}$ is injective on closed points.

\begin{lem}
For every sheaf $\mc{F}$ on $X$ we have
\[\mc{F}\simeq \mc{T}\!or^1_{\mc{O}_{X\times X}}(\mc{O}_\Delta,\sigma_\star\psi^\star\tamc{F}).\]
Thus we can reconstruct the original sheaf $\mc{F}$ from the tautological sheaf $\tamc{F}.$
\end{lem}

\begin{proof} Recall that we have an exact sequence on $X\times X$ (cf.\ Propostion \ref{eqtautfund}):
\begin{equation}\label{eqdeftautbasic}
0 \rightarrow \sigma_\star\psi^\star\tamc{F} \rightarrow \mc{F}^{\boxplus2}\rightarrow \Delta_\star\mc{F}\rightarrow0. 
\end{equation}
We tensor this sequence with the structure sheaf of the diagonal $\Delta\subset X\times X$. Of course we have
\[\pi_1^\star\mc{F}|_\Delta\simeq\Delta^\star\pi_1^\star\mc{F}\simeq\mc{F}\]
and the higher Tors $\mc{T}\!or^i_{\mc{O}_{X\times X}}(\mc{O}_\Delta,\pi_1^\star\mc{F})$ vanish. Therefore we have an isomorphism \[\mc{T}\!or^1_{\mc{O}_{X\times X}}(\mc{O}_\Delta,\sigma_\star\psi^\star\tamc{F})\simeq \mc{T}\!or^2_{\mc{O}_{X\times X}}(\mc{O}_\Delta,\Delta_\star\mc{F}).\]
By Proposition 11.8 in \cite{Huy06} we find
\[\hspace{80pt} \mc{T}\!or^i_{\mc{O}_{X\times X}}(\mc{O}_\Delta,\Delta_\star\mc{F})=\begin{cases} \mc{F} &i=0,2\text{ and}\\
\mc{F}\otimes\Omega_X &i=1.\hspace{130pt}\qedhere\end{cases}\]
\end{proof}

\begin{rem}
If we tensor (\ref{eqdeftautbasic}) with $\mc{O}_\Delta$ as above, the first terms of the resulting long exact Tor-sequence yield a short exact sequence
\[0\rightarrow \mc{F}\otimes\Omega_X\rightarrow \sigma_\star\psi^\star\tamc{F}|_\Delta\rightarrow \mc{F}\rightarrow 0.\]
It is not clear if this exact sequence is split or if it is equivalent to the natural extension corresponding to the Atiyah class of $\mc{F}.$
\end{rem}

Let us consider a stable sheaf $\mc{F}$ on a $K3$ surface $X$. The stability implies that either $h^0(X,\mc{F})$ or $h^2(X,\mc{F})=h^0(X,\mc{F}^\vee)$ vanishes. Let us assume the former is the case. (The case $h^2(X,\mc{F})=0$ can be treated in exactly the same way.) Corollary \ref{cortautext} shows that we have a natural monomorphism
\[\ta{[-]}\colon\Ext^1(\mc{F},\mc{F})\hookrightarrow\Ext^1(\tamc{F},\tamc{F}),\]
which maps an infinitesimal deformation of $\mc{F}$ to its induced deformation of $\tamc{F}.$

\begin{defi}
We call an infinitesimal deformation of $\tamc{F},$ the class of which lies in the image of $\ta{[-]}$ above, a \em surface deformation. \em Deformations lying in the other summand of equation (\ref{eqtautext}) in Corollary \ref{cortautext} are referred to as \em additional deformations. \em
\end{defi}

We conclude:
\begin{prop}\label{propdeftaut}
We have an embedding of moduli spaces
\[\mc{M}^s(v)\hookrightarrow\mc{M}^s(\ta{v}).\]
\end{prop}

The additional deformations are isomorphic to $\cH^0(X,\mc{F})\otimes \cH^1(X,\mc{F})^\vee.$

\begin{cor}
Let $\mc{F}$ be such that $h^1(X,\mc{F})=0.$ Then we have a local isomorphism of the corresponding moduli spaces.
\end{cor}

\begin{cor}
Let $\mc{F}$ be such that $\mc{M}^s(v)$ is compact and $h^1(X,\mc{G})=0$ for all $\mc{G}\in\mc{M}^s(v).$ Then we have an isomorphism of $\mc{M}^s(v)$ with a connected component of $\mc{M}^s(\ta{v}).$
\end{cor}

\subsection{The Additional Deformations and Singular Moduli Spaces} \label{subsectionmodspacsing}
In the last section we have seen that the surface deformations of tautological sheaves are unobstructed. This is not true for all deformations. Indeed, in this section we will give an explicit construction of an example of a sheaf $\mc{F}$ on an elliptically fibred $K3$ surface such that $\tamc{F}$ is stable and the corresponding point in the moduli space is singular.\\

To prove this statement let us recall the most basic properties of the Kuranishi map: The general idea of the deformation theory of a stable sheaf $\mc{F}$ is that infinitesimal deformations are parametrised by $\Ext^1(\mc{F},\mc{F})$ and the obstructions lie in $\Ext^2(\mc{F},\mc{F})$. This is formalised by the so-called Kuranishi map. More precisely it can be shown that there is a map $\kappa\colon\Ext^1(\mc{F},\mc{F})\rightarrow \Ext^2(\mc{F},\mc{F})$ such that the completion of the local ring of the point of the moduli space corresponding to $\mc{F}$ is isomorphic to the local ring of $\kappa^{-1}(0)$ in $0$. In general there is no direct geometric description of the Kuranishi map but it is known that the constant and linear terms of the power series expansion of $\kappa$ vanish and that its quadratic part is given by $\kappa_2\colon\Ext^1(\mc{F},\mc{F})\rightarrow \Ext^2(\mc{F},\mc{F}), e\mapsto \frac{1}{2}(e\circ e).$

For a $K3$ surface this quadratic term always vanishes since it is exactly the Serre duality pairing which is known to be alternating. But if we consider a tautological sheaf $\tamc{F}$ the quadratic part of the Kuranishi map may be non-trivial. This would correspond to the existence of a quadratic part in the equation of the tangent cone of the point in the moduli space corresponding to $\tamc{F}$. Consequently, the tangent cone would be strictly smaller than the tangent space and we would end up with a singularity.

\begin{exa}
Let $X$ be an elliptically fibred $K3$ surface with fibre class $E$ and section $C$. Consider the line bundle $\mc{G}:=\mc{O}(kF),$ $k\geq 2.$ We have $h^0(\mc{G})=k+1$ and $h^1(\mc{G})=k-1$. Certainly $\mc{G}$ is stable and the moduli space is a reduced point. The rank two tautological sheaf $\tamc{G}$ is also stable and the tangent space of its moduli space at the point corresponding to $\tamc{G}$ is isomorphic to $\cH^0(X,\mc{G})\otimes \cH^1(X,\mc{G})^\vee,$ which has dimension $k^2-1.$ The quadratic term of the Kuranishi map vanishes identically but it is not clear if we can deform $\tamc{G}$ along any of these infinitesimal directions.
\end{exa}

\begin{exa}
We continue with the same elliptic $K3$ as above. From \cite{DM89} we learn that the linear system of the line bundle $\mc{L}$ with first Chern class $C+kE$ has $C$ as a base component for $k\geq2$. We have $h^0(\mc{G})=k+1$ and $h^1(\mc{G})=0.$ Now let $p$ be a point on the curve $C$ and denote by $\mc{I}_p$ the corresponding ideal sheaf. We set $\mc{F}:=\mc{L}\otimes\mc{I}_p$. Certainly $\mc{F}$ is a torsion-free rank one sheaf with nonvanishing first Chern class. Hence $\tamc{F}$ is stable by Theorem \ref{thmregsurfcasesr1}.
\end{exa}

\begin{thm}
The point in the moduli space corresponding to $\tamc{F}$ is singular.
\end{thm}

By the above considerations we have to prove the following lemma:

\begin{lem}
For the example $\tamc{F}=\ta{(\mc{L}\otimes\mc{I}_p)}$ the quadratic part of the Kuranishi map does not vanish.
\end{lem}

\begin{proof} We have to analyse the Yoneda square
\begin{equation*}\begin{array}{ccc}\Ext^1(\tamc{F},\tamc{F})&\rightarrow &\Ext^2(\tamc{F},\tamc{F}).\\
x&\mapsto &x\circ x
\end{array}\end{equation*}
Therefore let us use Krug's formula (\ref{eqtautext}) in Corollary \ref{cortautext} to write down the extension groups explicitly. Note that $h^2(\mc{F})=0.$
\begin{equation*}\begin{array}{rcl}
\Ext^1(\tamc{F},\tamc{F})&\cong& \Ext^1(\mc{F},\mc{F})\bigoplus \cH^1(\mc{F})^\vee\otimes \cH^0(\mc{F}),\\[4pt]
\Ext^2(\tamc{F},\tamc{F})&\cong& \Ext^2(\mc{F},\mc{F})\bigoplus \cH^0(\mc{F})^\vee\otimes \cH^0(\mc{F})\bigoplus \cH^1(\mc{F})^\vee\otimes \cH^1(\mc{F}).
\end{array}\end{equation*}
According to this decomposition we can decompose the Yoneda square as well following the detailed formulas in \cite[Sect.\ 7]{Kru11}:
\begin{equation*}\begin{array}{ccccc}
\Ext^1(\mc{F},\mc{F})&\bigoplus& \cH^1(\mc{F})^\vee\otimes \cH^0(\mc{F})&\rightarrow& \\
e&+&a\otimes b&\mapsto&\vspace{0.5cm}\\

\Ext^2(\mc{F},\mc{F})&\bigoplus& \cH^0(\mc{F})^\vee\otimes \cH^0(\mc{F})&\bigoplus& \cH^1(\mc{F})^\vee\otimes \cH^1(\mc{F}).\\
\underbrace{e\circ e}_{=0}&+&(a\circ e)\otimes b&+&a\otimes(e\circ b)
\end{array}\end{equation*}
Hence we need to show that the map
\[\Ext^1(\mc{F},\mc{F})\times \cH^0(\mc{F})\rightarrow \cH^1(\mc{F})\]
is not the zero map. The geometric interpretation of this map is the following: Let $e\in\Ext^1(\mc{F},\mc{F})$ be an infinitesimal deformation of $\mc{F}$ and $\varphi\in \cH^0(\mc{F})$ be a global section. Then $\varphi\circ e$ is zero if and only if we can deform the section $\varphi$ along $e$.

It is time to return to the geometry of our example. Since $p$ is on the base curve $C$, we have $\cH^0(\mc{F})\cong \cH^0(\mc{L})$. The deformations of $\mc{F}$ are those of $\mc{I}_p,$ which correspond to deforming the point $p$ in $X$. Now if we deform $p$ into a direction normal to $C$, the space of global sections will shrink since the point will fail to be a base point of $\mc{L}$ and thus we can find a section $\varphi\in \cH^0(\mc{F})$ that does not deform with $e$. \end{proof}

The Zariski tangent space is $(k+3)$-dimensional and we can explicitly derive the quadratic equation of the tangent cone. It is equivalent to the intersection of a plane (corresponding to the surface deformations) and a hyperplane (the additional deformations and the curve $C$) in a line (the curve $C$).

\subsection{Deformations of the Manifold $\tn{X}$} \label{subsectiondefmani}
A question which has not been touched so far, is the following: The manifold $\tn{X}$ has an unobstructed deformation theory. Does the tautological sheaf $\tnmc{F}$ deform with $\tn{X}$?

The technique to answer this question is presented in \cite{HT10}. We can summarise as follows:

\begin{thm}[Huybrechts$-$Thomas]
Let $Y$ be a projective manifold and $\mc{E}$ a sheaf on $Y.$ Let $\kappa\in \cH^1(Y,\mc{T}_Y)\cong\Ext^1(\Omega_Y,\mc{O}_Y)$ be the Kodaira$-$Spencer class of an infinitesimal deformation of $Y$ and denote by $\At(\mc{E})\in \Ext^1(\mc{E},\mc{E}\otimes\Omega_Y)$ the Atiyah class of $\mc{E}.$ The sheaf $\mc{E}$ can be deformed along $\kappa$ if and only if
\[0=\mathrm{ob}(\kappa,\mc{E}):=(\kappa\otimes\id_\mc{E})\circ \At(\mc{E})\in\Ext^2(\mc{E},\mc{E}).\]
\end{thm}

For every sheaf $\mc{E}$ on $Y$ there are natural trace maps
\begin{eqnarray*}
\tr\colon\Ext^1(\mc{E},\mc{E}\otimes\Omega_Y)&\rightarrow&\cH^1(Y,\Omega_Y)\quad\text{and}\\
\tr\colon\Ext^2(\mc{E},\mc{E})&\rightarrow&\cH^2(Y,\mc{O}_Y),
\end{eqnarray*}
which --- up to a sign --- commute with the Yoneda product. Furthermore, it is well known that
\begin{eqnarray*}
\tr(\At(\mc{E}))&=&c_1(\mc{E})\quad\text{and}\\
\tr(\mathrm{ob}(\kappa,\mc{E}))&=&\mathrm{ob}(\kappa,\det\mc{E})
\end{eqnarray*}

Applying this theorem to our situation, we get the following picture:
The tangent space of the Kuranishi space at the point corresponding to $\tn{X}$ is isomorphic to

\[\cH^1(\tn{X},\mc{T}_{\tn{X}})\cong \cH^1(\tn{X},\Omega_{\tn{X}})\cong \cH^1(X,\Omega_X)\oplus \CC\delta_n.\]

We write a class in $\cH^1(\tn{X},\Omega_{\tn{X}})$ as $(\kappa,a)$ with $\kappa\in \cH^1(X,\Omega_X)$ the class of an infinitesimal deformation of the surface $X$ and $a\in\CC$. Unfortunately there is no decomposition of the Atiyah class $\At(\tnmc{F})$ at hand. But we can at least study its trace $\tr(\At(\tnmc{F}))=c_1(\tnmc{F})=c_1(\mc{F})_{\tn{X}}-r\delta_n,$ where we set $r:=\rk\mc{F}.$ We have:
\[\tr(\mathrm{ob}((\kappa,a),\tnmc{F}))=\mathrm{ob}(\kappa,\det\mc{F})-ra\delta^2_n\in \cH^2(\tn{X},\mc{O}_{\tn{X}}).\]
But we have $\mathrm{ob}(\kappa,\det\mc{F})=\kappa\cdot c_1(\mc{F})$ and $\delta^2_n=2(1-n),$ where we consider the Beauville$-$Bogomolov pairing.
Thus we see:
\begin{itemize}
\item If $\mc{F}$ deforms along $\kappa$, then surely the tautological sheaf $\tnmc{F}$ deforms along $(\kappa,0)$.
\item If the determinant line bundle $\det\mc{F}$ does not deform along $\kappa$, then $\tnmc{F}$ does not deform along $(\kappa,0)$.
\item If $\kappa\cdot c_1(\mc{F})\neq 2(1-n)ra$, the tautological sheaf $\tnmc{F}$ does not deform along $(\kappa,a)$.
\end{itemize}

Thus there is an interesting hyperplane inside the space of infinitesimal deformations of $\tn{X}$ consisting of all pairs $(\kappa,a)$ such that $\kappa\cdot c_1(\mc{F})=2(1-n)ra$: It is an open question if the tautological sheaf deforms along these directions.


\newpage
\addcontentsline{toc}{section}{Bibliography}

\begin{thebibliography}{einszweid}
\bibitem[Ara]{Ara} D.\ Arapura,
	\em Abelian Varieties and Moduli, \em 
	http://www.math.purdue.edu/\~{} dvb/
	
\bibitem[Ati57]{Ati57} M. Atiyah, 
	\em Complex analytic connections in fibre bundles, \em
	Trans.\ Amer.\ Math.\ Soc.\ \textbf{85} (1957), 181-207. 
	
\bibitem[Beau83]{Beau83}A.\ Beauville,
	\em Vari\'{e}t\'{e}s K\"ahleriennes dont la premi\`{e}re classe de Chern est nulle, \em
	J.\ Diff.\ Geom.\ \textbf{18} (1983), 755-782.

\bibitem[BHPV04]{BHPV04}{W.\ Barth, K.\ Hulek, C.\ Peters, A.\ Van de Ven},
	{\it Compact Complex Surfaces (Second Enlarged edition),\it}
	{Erg.\ der Math.\ und ihrer Grenzgebiete, 3.\ Folge, Band 4, Springer, 2004}.


\bibitem[BL92]{BL92} G.\ Birkenhake, H.\ Lange, 
	\em Complex abelian varieties, \em 
	Grundlehren Math.\ Wiss.\ \textbf{302}, Springer, 1992.

\bibitem[EGL01]{EGL01} G.\ Ellingsrud, L.\ G\"ottsche, M.\ Lehn,
	\em On the cobordism class of the Hilbert scheme of a surface, \em
	Journal of Algebraic Geometry \textbf{10} (2001), 81-100.

\bibitem[ES98]{ES98} G.\ Ellingsrud, S.\ A.\ Str\o mme,
	\em An intersection number for the punctual Hilbert scheme of a surface, \em
	Trans.\ Amer.\ Math.\ Soc.\ \textbf{350}, no.\ 6 (1998), 2547-2552.

\bibitem[GH98]{GH98} V. Gritsenko, K. Hulek,
	\em Minimal Siegel modular threefolds, \em
	Math. Proc. Cambridge Philos. Soc. \textbf{123}, no.\ 3 (1998), 461–485.

\bibitem[Hai01]{Hai01} M.\ Haiman,
	\em Hilbert schemes, polygraphs and the Macdonald positivity conjecture, \em
	Journal of the American Mathematical Society \textbf{14}, no.\ 4 (2001), 941-1006 (electronic).

\bibitem[HL97]{HL97}{D.\ Huybrechts, M.\ Lehn},
	{\it The geometry of moduli spaces of sheaves,\it}
	{Aspects of Mathematics 31, Friedr.\ Vieweg \& Sohn, Braunschweig 1997.}

\bibitem[Huy99]{Huy99} D. Huybrechts,
	\em Compact hyperk\"ahler manifolds: basic results, \em
	Invent. Math. \textbf{135} (1999), 63-113. Erratum in: Invent. Math. \textbf{152} (2003), 209-212.

\bibitem[Huy06]{Huy06} D. Huybrechts, 
	\em Fourier$-$Mukai Transforms in Algebraic Geometry, \em 
	Oxford Mathematical Monographs, 2006.

\bibitem[KKV89]{KKV89} F. Knop, H. Kraft, Th. Vust, 
	\em The Picard group of a G-variety, \em 
	Algebraische Transformationsgruppen und Invariantentheorie, (H. Kraft, P. Slodowy, T. Springer eds.) DMV Semin. \textbf{13}, Basel-Boston-Berlin: Birkh\"auser 1989, 77–88.

\bibitem[Kru11]{Kru11}{A. Krug},
	{\it Extension groups of tautological sheaves on Hilbert schemes,\it}
	alg-geom/1111.4263, (2011), 49 pages.

\bibitem[Laz04]{Laz04} R. Lazarsfeld,
	\em Positivity in algebraic geometry I, \em
	Ergebnisse der Mathematik und ihrer Grenzgebiete, 3. Folge, vol. 48, Springer-Verlag, Berlin, 2004.


\bibitem[Scal09b]{Scal09b} L.\ Scala,
	\em Some remarks on tautological sheaves on Hilbert schemes of points on a surface, \em
	Geom. Dedicata, Vol. 139, no. 1 (2009), 313-329.

\bibitem[Sch10]{Sch10} U. Schlickewei,
	\em Stability of tautological vector bundles on Hilbert squares of surfaces,  \em
	Rendiconti del Seminario Matematico della Universit\'{a} di Padova \textbf{124} (2010).

\bibitem[Wan12]{Wan12} M.\ Wandel,
	\em Stability of tautological bundles on the Hilbert scheme of two points on a surface, \em to appear in Nag.\ Math.\ J.\, alg-geom/1202.6528 (2012), 14 pages.




\end{thebibliography}
\end{document}